\documentclass[12pt]{amsart}

\allowdisplaybreaks[4]
\DeclareMathAlphabet{\mathcal}{OMS}{cmsy}{m}{n}

\usepackage{epsfig}
\usepackage{graphics}
\usepackage{amsfonts}
\usepackage{amsmath}
\usepackage{amscd,amssymb,mathrsfs}
\usepackage{latexsym}

\newtheorem{theorem}{Theorem}[section]
\newtheorem{lemma}{Lemma}[section]
\newtheorem{proposition}{Proposition}[section]
\newtheorem{corollary}{Corollary}[section]

\newtheorem{remark}{Remark}[section]

\begin{document}
\title[Singularity spectra of $b$-adic independent cascade function]{The graph, range and level set singularity spectra of $b$-adic independent cascade function}

\author{Xiong Jin}
\address{INRIA Rocquencourt, B.P. 105, 78153 Le Chesnay Cedex, France}
\email{xiongjin82@gmail.com}

\begin{abstract}
With the ``iso-H\"older" sets of a function we naturally associate subsets of the graph, range and level set of the function. We compute the associated singularity spectra for a class of statistically self-similar multifractal functions, namely the $b$-adic independent cascade function.
\end{abstract}

\maketitle

\section{Introduction}
\label{intro}

\subsection{The singularity spectra of multifractal function.}

Let $f$ be a real-valued function defined on an interval $I$. For any subset $E$ of $I$ let
\begin{equation*}
G_f(E)=\{(x,f(x)):x\in E\},\ R_f(E)=\{f(x):x\in E\}
\end{equation*}
be the graph and the range of $f$ over the set $E$; for any $y\in R_f(E)$ let
\begin{equation*}
L_{f}^{y}(E)=\{(x,f(x)): x\in E, f(x)=y\}
\end{equation*}
be the level set of $f$ over the set $E$ at level $y$. When $(f(t))_{t\in I}$ is a stochastic process, to find the Hausdorff dimension of these sets, denoted by $\dim_{H} S_f(E)$ for $S\in\{G,R,L^y\}$, is a classical and important question in probability and geometric measure theory. The original works on these questions could be traced back to 1953,~\cite{Levy} by L\'evy or~\cite{Tay} by Taylor, regarding the Hausdorff dimension and the Hausdorff measure of the range of Brownian motion. Since then, many progresses have been made in this subject for fractional Brownian motions, stable L\'evy processes and many other processes and functions \cite{BlGe1961,BlGe1962,TaWe1966,Horo1968,Pru1969,Mi1971,Orey1971,Ber1972,Haw1974,Ber1983,Ka1985b,MaWi1986,PrzUr1989,Ur1990b,BeUr1990,Bert1990,Led1992,HuLau1993,Hunt1998,Ro2003b,Fe2005,KhXi2005,DeFa2007} (see also the survey paper~\cite{Xiao2004} and the references therein).

As a typical example, in~\cite{Ka1985b} Kahane studies the fractional Brownian motion $(X(t))_{t\in\mathbb{R}_+}$, i.e., given $\beta\in(0,1)$, the unique centered continuous Gaussian process satisfying $X(0)=0$ and $\mathbb{E}(|X(s)-X(t)|^2)=|s-t|^{2\beta}$ for any $s,t\in \mathbb{R}_+$. He shows that for any compact set $E\subset \mathbb{R}_+$, almost surely
\begin{equation*}\label{upG}
\dim_H G_X(E) = \frac{\dim_H E}{\beta} \wedge (\dim_H E+1-\beta),\ \dim_H R_X(E)  = \frac{\dim_H E}{\beta} \wedge 1
\end{equation*}
and if $\dim_H E>\beta$, then there exists a random open set $G \subset R_X(E)$ such that $\mathbb{P}(G\neq\emptyset)>0$ and
\begin{equation*}\label{upL}
y\in G \Rightarrow \dim_H L_{X}^{y}(E) = \dim_H E-\beta.
\end{equation*}

Notice that here $\beta\in(0,1)$ is the (single) H\"older exponent of the function $X$ and it is uniform on $E$, indeed $X$ is a monofractal on $\mathbb{R}_+$. When the function is not monofractal, a natural parallel to the above formulas is to take $E=E_f(h)$, the set of points at which the pointwise H\"older exponents of $f$ are all equal to a constant $h>0$, and to verify formulas like
\[
\dim_H G_f(E_f(h)) = \frac{\dim_H E_f(h)}{h} \wedge \big(\dim_H E_f(h)+1-h\big)
\]
and
\[
\dim_H R_f(E_f(h))  = \frac{\dim_H E_f(h)}{h} \wedge 1.
\]
The set $E_f(h)$ and its Hausdorff dimension $\dim_H E_f(h)$ naturally appeared in the multifractal analysis of functions, which consists in computing the singularity spectrum $d_f: h\ge 0\mapsto \dim_H E_f(h)$. This, together with the above parallel formulas, leads us to consider the following graph, range and level set singularity spectra:
\begin{equation*}
d_f^{S}:h\ge 0\mapsto \dim_H S_f(E_f(h)),\ \ S\in\{G,R,L^y\}.
\end{equation*}
To our best knowledge, such singularity spectra have not been considered before.

\subsection{A general upper bound}

At first, it is natural to seek for general upper bounds for these new singularity spectra. Such bounds can be found thanks to the following generalization of Lemma 8.2.1 in~\cite{Adl}, Theorem 6 of Chapter 10 in~\cite{Ka1985b} and Lemma 2.2 in~\cite{Xiao1}.

First we note that the pointwise H\"older exponent considered in this paper is defined by
\begin{equation}\label{hf}
h_f(x):=\liminf_{r\to 0^+} \frac{1}{\log r} \log \left(\sup_{s,t\in B(x,r)}|f(s)-f(t)|\right),
\end{equation}
which covers the definition of the lower local dimension
\begin{equation}\label{hmu}
h_\mu(x):=\liminf_{r\to 0^+} \frac{1}{\log r}\log \mu\Big(B(x,r)\Big)
\end{equation}
of a measure $\mu$ support on an interval $[a,b]$ if we set $f(x)=\mu([a,x])$ for $x\in[a,b]$.

\medskip

Also, we should introduce the level set of $f$ over set $E$ in $\theta$-direction (since in this paper we can only show the level set singularity spectrum in ``Lebesgue almost every direction"): For $\theta\in(-\pi/2,\pi/2)$ denote by $l_\theta$ the line in $\mathbb{R}^2$ passing through the origin and making an angle $\theta$ with the $y$-axis (clockwise). For any $y\in l_\theta$, denote by $l_{y,\theta}^{\perp}$ the line perpendicular to $l_\theta$,  passing through $y$. Denote by $\mathrm{Proj}_\theta$ the orthogonal projection from $\mathbb{R}^2$ onto $l_\theta$. Define $R_{f,\theta}(E)=\mathrm{Proj}_\theta(G_f(E))$. Then for each $y\in R_{f,\theta}(E)$, the level set of $f$ over the set $E$ in $\theta$-direction is defined by $L^y_{f,\theta}(E)=G_f(E)\cap l_{y,\theta}^{\perp}$. Notice that the typical level set $L^y_f(E)$ is just the level set $L^y_{f,\theta}(E)$ when $\theta=0$.

\medskip

Denote by $\dim_P$ the packing dimension. We have the following theorem:

\begin{theorem}\label{upper}

Let $E$ be any subset of $I$. Suppose that $\inf_{x\in E} h_f(x)=h>0$.

\medskip

\begin{itemize}
\item[(a)] For $D\in\{H,P\}$ we have
\begin{eqnarray*}
\dim_D G_f(E) &\le& \Big( \frac{\dim_D E}{h}\wedge \big(\dim_D E+1-h\big) \Big) \vee \dim_D E, \\
\dim_{D} R_f(E) &\le&  \frac{\dim_D E}{h} \wedge 1.
\end{eqnarray*}
\item[(b)] Suppose $h\le1$. Fix $\theta\in(-\pi/2,\pi/2)$. Let $\mu$ be any positive Borel measure defined on $l_\theta$. For any $\gamma>0$ define the set $R_{f,\theta}^{\mu,\gamma}(E):=\{y\in R_{f,\theta}(E): h_\mu(y) \ge \gamma\}$. If $\mu(R_{f,\theta}^{\mu,\gamma}(E))>0$ and $\dim_H E-h\cdot \gamma>0$, then for $\mu$-almost every $y\in R_{f,\theta}^{\mu,\gamma}(E)$,
$$
\dim_H L_{f,\theta}^{y}(E) \le \dim_H E-h\cdot \gamma.
$$
\end{itemize}
\end{theorem}

If we replace $E$ by the set $E_f(h)=\{x\in I: h_f(x)=h\}$ for $h>0$, then Theorem~\ref{upper} provides us with general upper bounds of the graph, range and level set singularity spectra. These upper bounds are strongly related to the classical singularity spectrum $d_f$ (see Corollary \ref{upperspec}).

\medskip

From the multifractal analysis of functions we know that $d_f$ has a general upper bound given by the Legendre transform of the so-called scaling function or $L^q$-spectrum of $f$, defined by
\begin{equation}\label{tauq}
\tau_f(q)=\liminf_{r\to 0^+} \frac{1}{\log r} \log \sup \sum_i \mathrm{Osc}_f(B_i)^q,\ q\in\mathbb{R},
\end{equation}
where $\mathrm{Osc}_f(B_i)=\sup_{s,t\in B_i}|f(s)-f(t)|$ denotes the oscillation of $f$ over $B_i$ and the supremum is taken over all the families of disjoint closed intervals $B_i$ of radius $r$ with centers in $\{x\in I: \forall\ r>0,\ \mathrm{Osc}_f(B(x,r))>0\}$. Due to \cite{Jafw,BJ}, we have
\begin{equation*}\label{ieqspec}
d_{f}(h) \le \tau_f^*(h):=\inf_{q\in\mathbb{R}}qh-\tau_f(q)\quad (\forall h\ge 0),
\end{equation*}
a negative dimension meaning that $E_f(h)$ is empty. Then as a direct consequence, we have the following corollary of Theorem~\ref{upper}:

\begin{corollary}\label{upperspec}
For any $h>0$ we have
\begin{eqnarray*}
\dim_H G_f(E_f(h)) &\le&  \Big(\frac{d_{f}(h)}{h} \wedge \big(d_{f}(h)+1-h\big) \Big) \vee d_{f}(h) \\
&\le& \Big( \frac{\tau_{f}^*(h)}{h} \wedge \big(\tau_{f}^*(h)+1-h\big) \Big) \vee \tau_{f}^*(h),\\
\dim_H R_f(E_f(h)) &\le&  \frac{d_{f}(h)}{h}  \wedge 1 \le \frac{\tau_{f}^*(h)}{h} \wedge 1
\end{eqnarray*}
and with the same notations as in Theorem~\ref{upper}(b), for $\mu$-almost every $y\in R_{f,\theta}^{\mu,\gamma}(h)$,
\begin{equation*}
\dim_H L_{f,\theta}^y(E_f(h)) \le d_f(h)-h \cdot \gamma \le \tau_f^*(h)-h\cdot \gamma.
\end{equation*}
\end{corollary}

\subsection{Main result} Corollary~\ref{upperspec} naturally raises the question: Do these upper bounds provide the exact dimensions, especially when $f$ obeys the multifractal formalism, i.e. $d_f(h)=\tau_f^*(h)$ for $h> 0$ ?

\medskip

In general the answer is negative. We can easily find a counterexample in case when $f$ is a monofractal function, that is $h_f(x)$ is equal to a constant $\beta\in(0,1)$ for all $x\in I$. Suppose, moreover, that the whole graph $G_f(I)=G_f(E_f(\beta))$ is irregular, that is its Hausdorff and lower box-counting dimensions are different. We have:
$$1-\beta+\tau_f^{*}(\beta)= 1-\beta+d_{f}(\beta)=2-\beta \ge \underline{\dim}_{B}G_f(E_f(\beta)) >\dim_H G_f(E_f(\beta)).$$
Such examples can be found in~\cite{Ur1990b,Fe2005}. To the contrary, if the whole graph is regular (like for the fractional Brownian motion mentioned before), then $\dim_H G_f(E_f(\beta))=1-\beta+\tau_f^{*}(\beta)$.

However, monofractal examples clearly represent a very restrictive class for our purpose. Simple multifractal examples are the following: Consider $f(x)=\mu([0,x])$ for $x\in[0,1]$, where $\mu$ is a probability measure fully supported by $[0,1]$, and assume that $f$ obeys the multifractal formalism with the exponent $\tilde{h}_f(x)$ defined by
\begin{equation}\label{limex}
\tilde{h}_f(x)=\lim_{r\to 0^+} \frac{1}{\log r} \log \Big(\sup_{s,t\in[x-r,x+r]}|f(s)-f(t)|\Big),
\end{equation}
whenever it exists. This property holds whenever $\mu$ is a Gibbs or a random cascade measure \cite{BMP,B2}. Then by using the results in~\cite{MaRi} on the multifractal analysis of the inverse measure $\mu^*=\mu\circ f^{-1}$ carried by the range of $f$, it is easy to check that the upper bounds in Corollary~\ref{upperspec} give the exact dimensions. But there, the graph and range singularity spectra are always a combination of the singularity spectra of $\mu$ and $\mu^*$, and the level set spectrum is trivial since $f$ is an increasing function.

\medskip

It is an interesting question to find examples of multifractal functions whose graph, range and level set singularity spectra can be calculated and are not trivial in the above sense. In this paper we consider the so-called $b$-adic independent cascade function introduced in~\cite{BMJ} as an extension to random functions of statistically self-similar measures introduced in \cite{Mand}. We obtain the graph and range singularity spectra for this class of random functions, in the so-called non-conservative case. The level set singularity spectrum is still open. However, inspired by the classical Marstrand theorem (see \cite{Mar1954}), we obtain the level set singularity spectrum in Lebesgue almost every direction.

\medskip

In a brief, the result is the following: Let $F$ be the $b$-adic independent cascade function (see Figure~\ref{Fpic} for an illustration and see Section \ref{badicF} for definition). For $h>0$ denote $S_F(h)=S_F(E_F(h))$ for $S\in\{G,R\}$. For $\theta\in(-\pi/2,\pi/2)$, recall that $l_\theta$ is the line in $\mathbb{R}^2$ passing through the origin and making an angle $\theta$ with the $y$-axis, and for any $y\in l_\theta$, $l_{y,\theta}^{\perp}$ is the line perpendicular to $l_\theta$,  passing through $y$. For $h> 0$ define $R_{F,\theta}(h)=\mathrm{Proj}_\theta(G_F(h))$, and for each $y\in R_{F,\theta}(h)$ define $L^y_{F,\theta}(h)=G_F(h)\cap l_{y,\theta}^{\perp}$.

\medskip

We have the following theorem (the assumptions (A1)-(A3) will be given later):

\begin{theorem}\label{mainthm}

Suppose that assumptions (A1)-(A3) hold.

\medskip

\begin{itemize}
\item[(a)]
Almost surely for all $h\in J_F=\{h> 0: \tau^*_F(h)>0\}$,
\begin{eqnarray*}
\dim_H G_F(h) &=& \Big( \frac{\tau_{F}^*(h)}{h} \wedge \big(\tau_{F}^*(h)+1-h\big) \Big) \vee \tau_{F}^*(h),\\
\dim_H R_F(h) &=& \frac{\tau^*_F(h)}{h} \wedge 1.
\end{eqnarray*}
Moreover, denote $G_F$ the whole graph, then almost surely
\[
\dim_H G_F = \dim_P G_F = \dim_B G_F =1-\tau_F(1).
\]

\item[(b)]
Almost surely for Lebesgue almost every $\theta\in (-\pi/2,\pi/2)$, for all $h\in (0,1)$ such that $\tau_F^*(h)-h>0$, for $\mu^{R}_{h,\theta}$ almost every $y\in R_{F,\theta}(h)$,
\[
\dim_H L^y_{F,\theta}(h)= \tau_F^*(h)-h,
\]
where $\mu_{h,\theta}^R$ is a positive Borel measure carried by $R_{F,\theta}(h)$ and it is absolutely continuous with respect to the one-dimensional Lebesgue measure on $l_\theta$.
\end{itemize}
\end{theorem}

\begin{figure}[b]
\begin{center} 
\includegraphics[width=\textwidth]{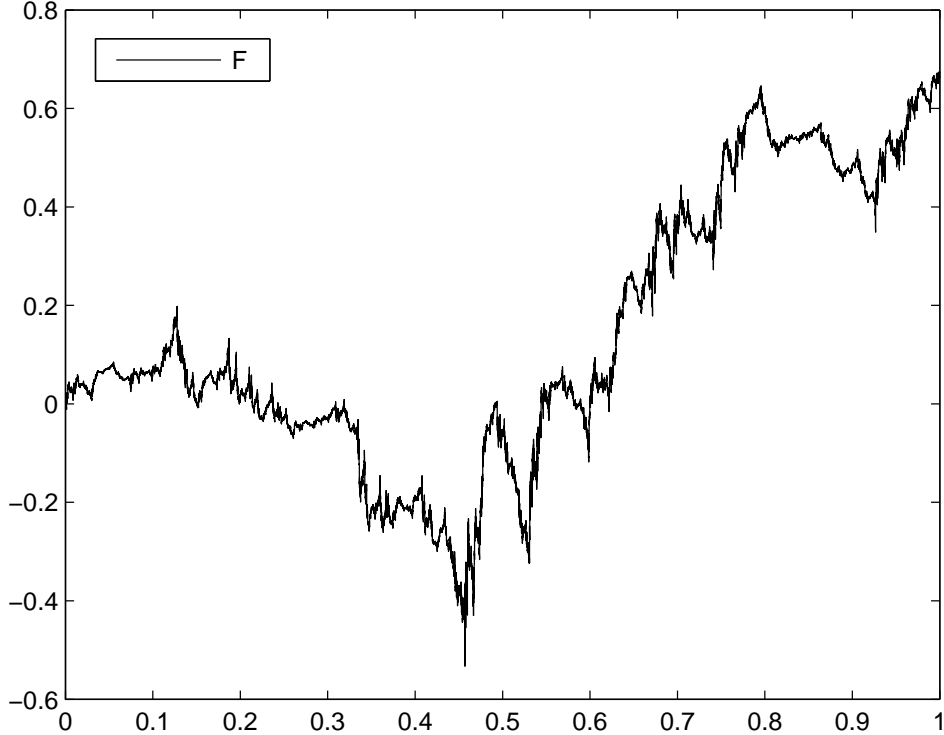} 
\end{center}
\caption{$F$ in the non-conservative case.}\label{Fpic}
\end{figure}

\begin{figure}
\begin{center}
\begin{minipage}[b]{.5\textwidth} 
\centering 
\includegraphics[width=\textwidth]{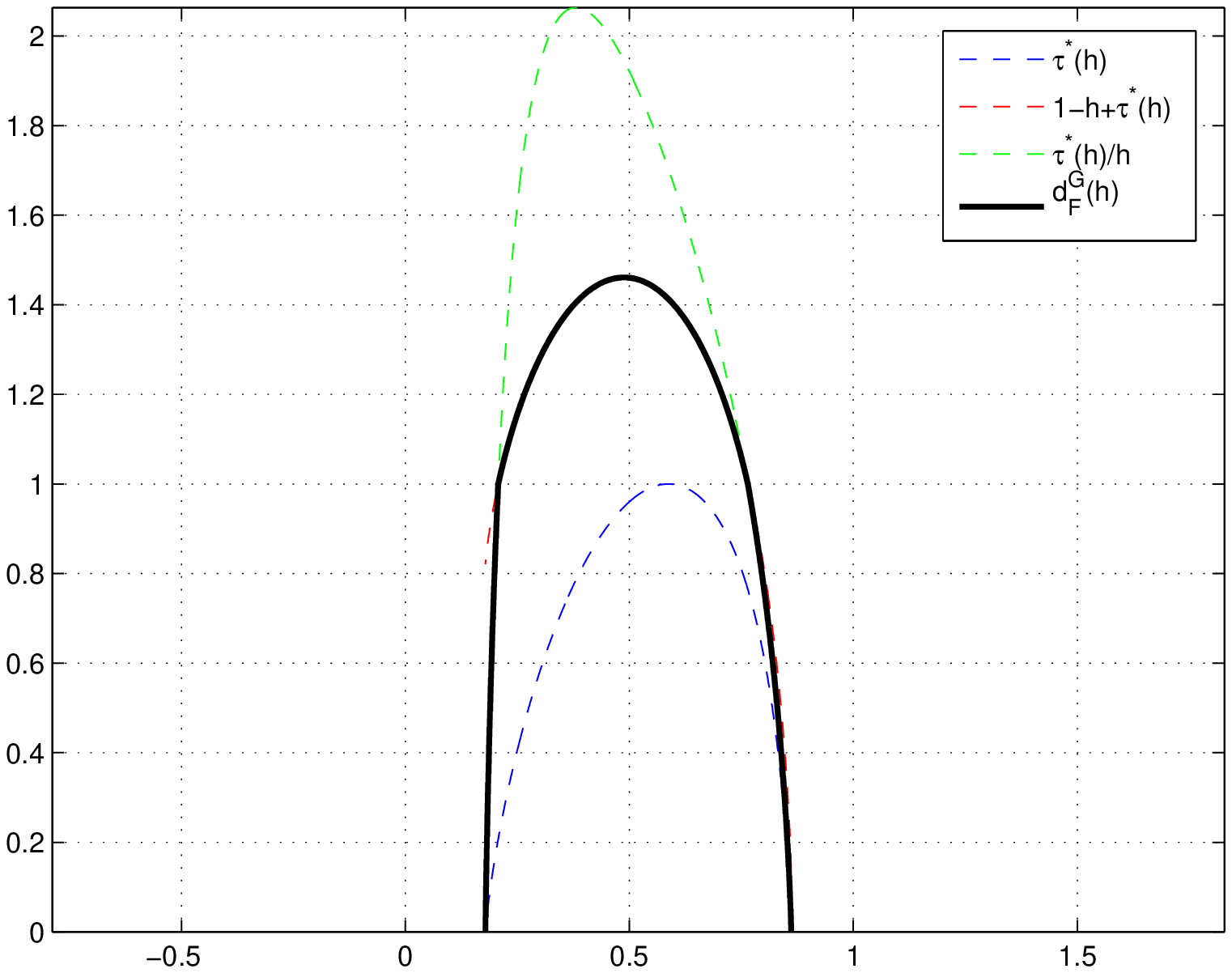} 
\par\vspace{0pt} 
\end{minipage}%
\begin{minipage}[b]{.5\textwidth} 
\centering 
\includegraphics[width=\textwidth]{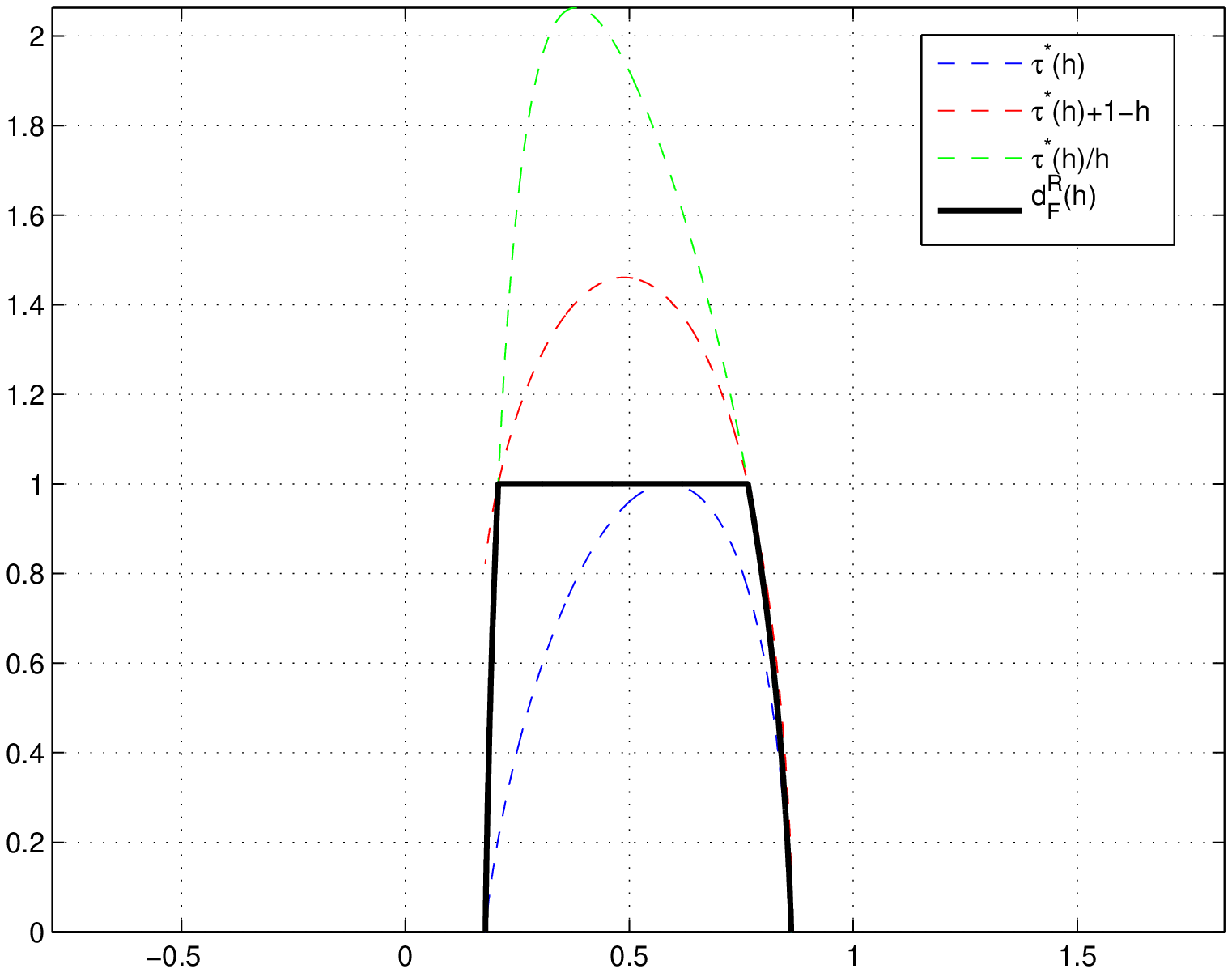}
\end{minipage}
\end{center}
\caption{$d_F^G$ (Left) and $d_F^R$ (Right) in case of $\sup J_F<1$.}
\end{figure}

\begin{figure}
\begin{center}
\begin{minipage}[b]{.5\textwidth} 
\centering 
\includegraphics[width=\textwidth]{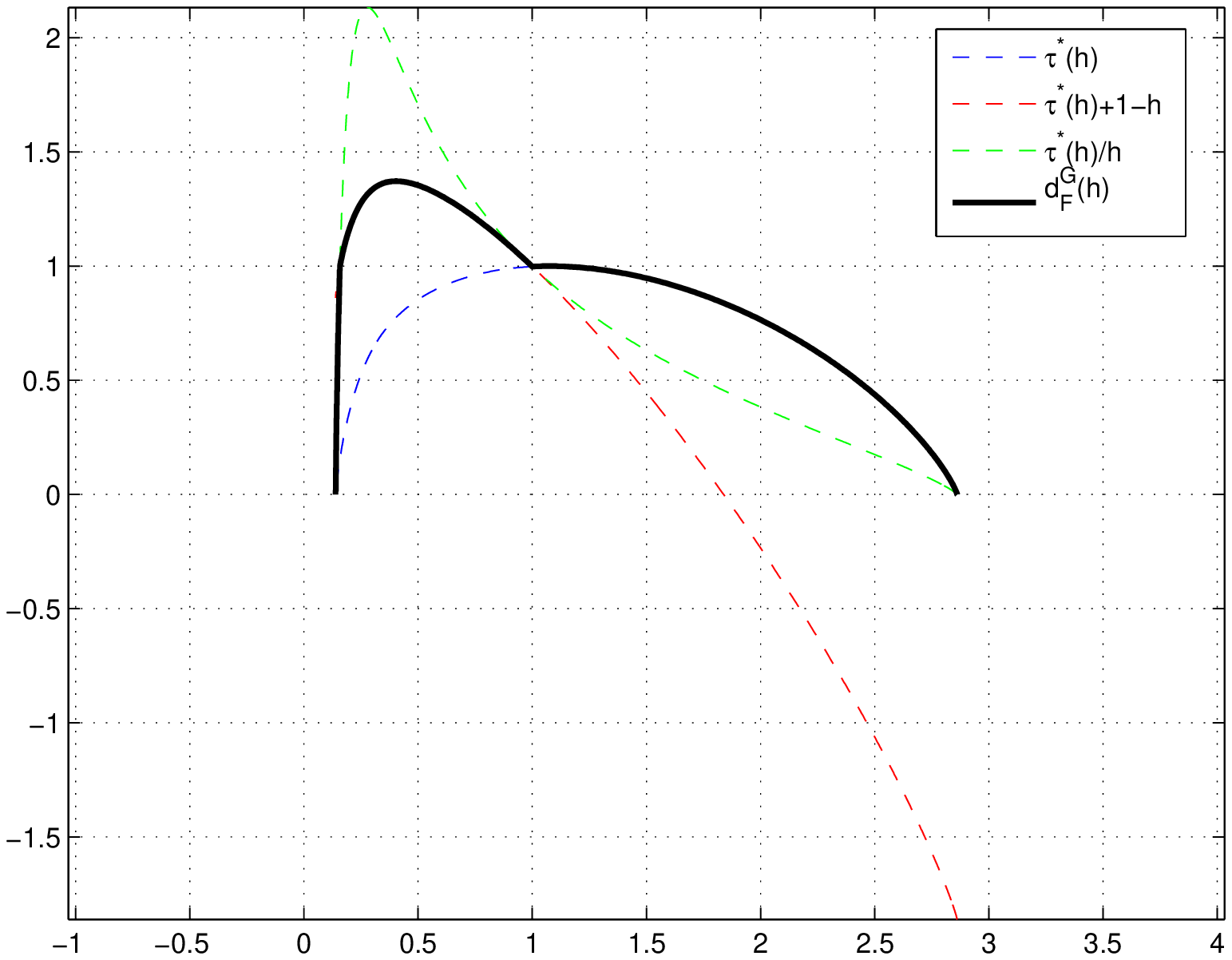} 
\par\vspace{0pt} 
\end{minipage}%
\begin{minipage}[b]{.5\textwidth} 
\centering 
\includegraphics[width=\textwidth]{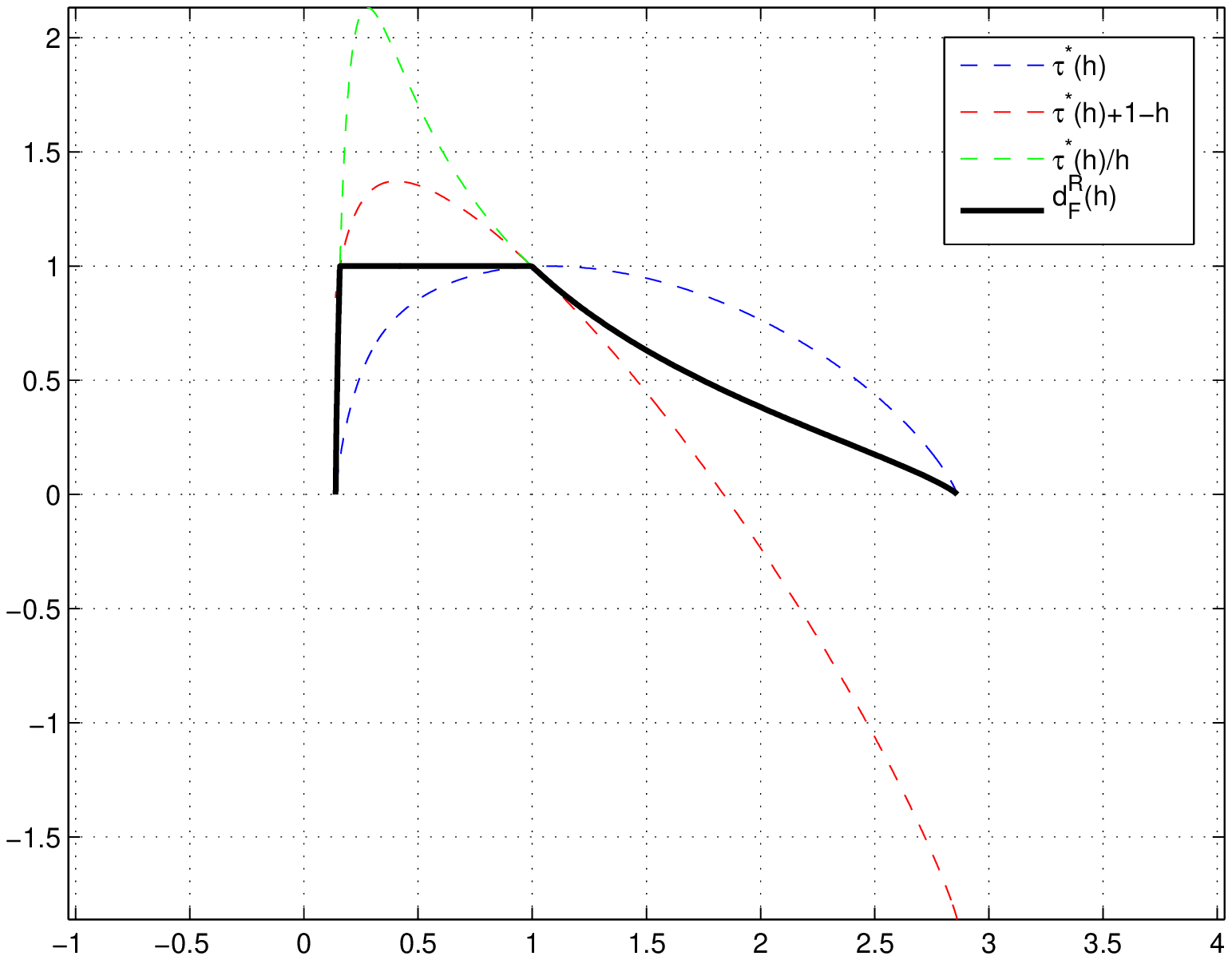} 
\end{minipage}
\end{center}
\caption{$d_F^G$ (Left) and $d_F^R$ (Right) in case of $\sup J_F>1$.}
\end{figure}

\begin{figure}
\begin{center}
\begin{minipage}[b]{.5\textwidth} 
\centering 
\includegraphics[width=\textwidth]{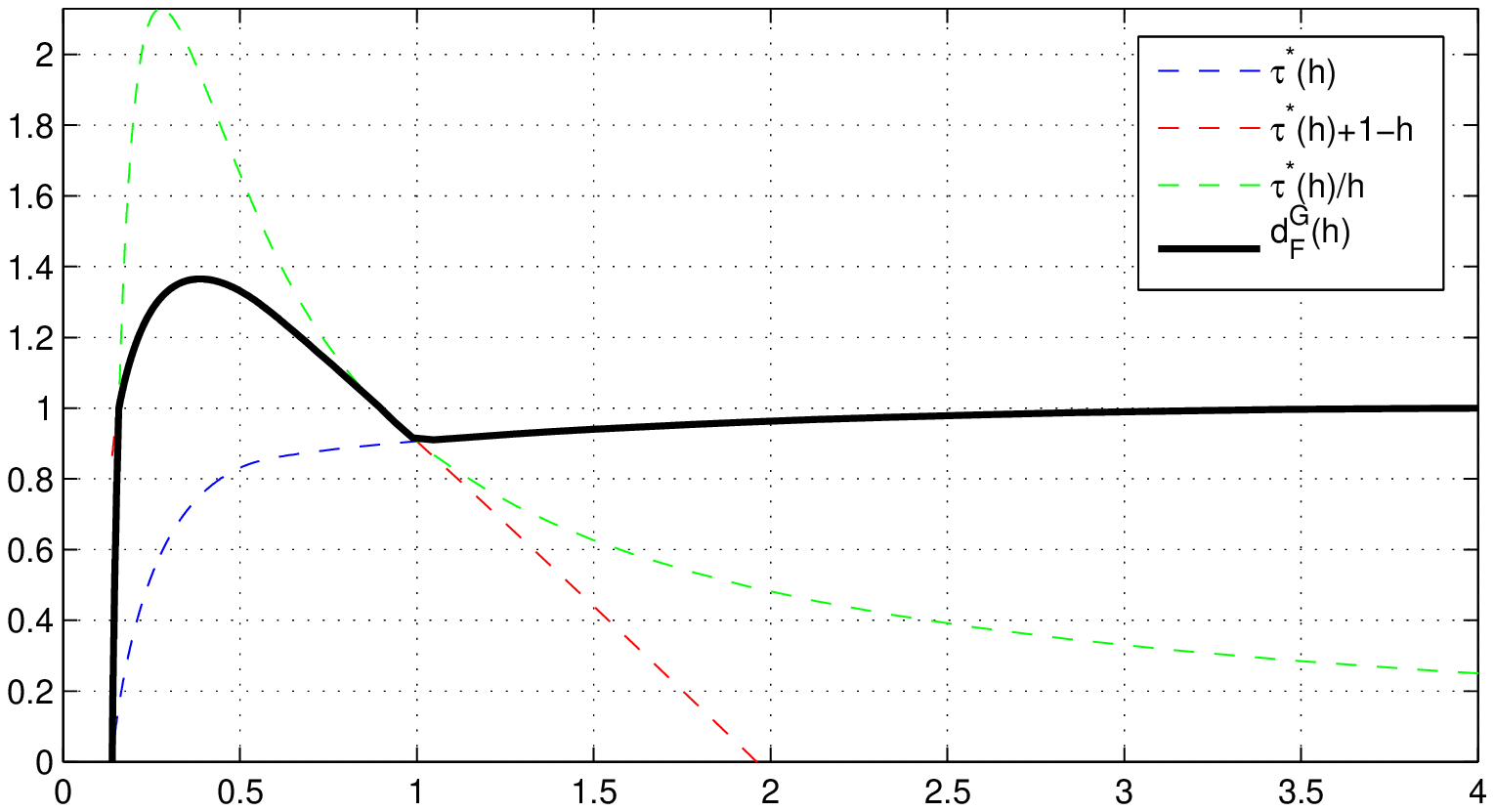} 
\par\vspace{0pt} 
\end{minipage}%
\begin{minipage}[b]{.5\textwidth} 
\centering 
\includegraphics[width=\textwidth]{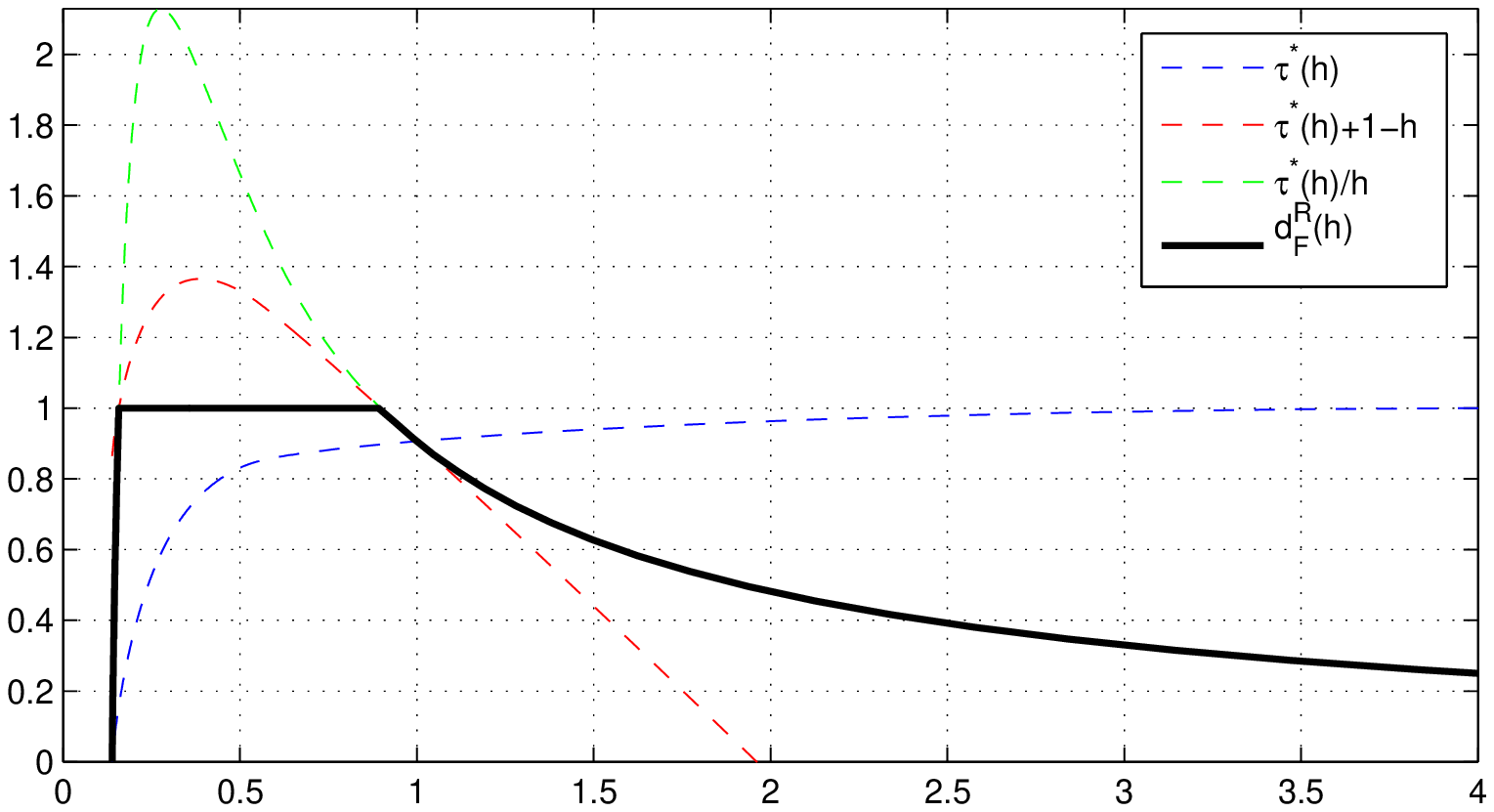} 
\end{minipage}
\end{center}
\caption{$d_F^G$ (Left) and $d_F^R$ (Right) in case of $\sup J_F=\infty$.}
\end{figure}

\begin{figure}
\begin{center} 
\begin{minipage}[b]{.33\textwidth} 
\centering 
\includegraphics[width=\textwidth]{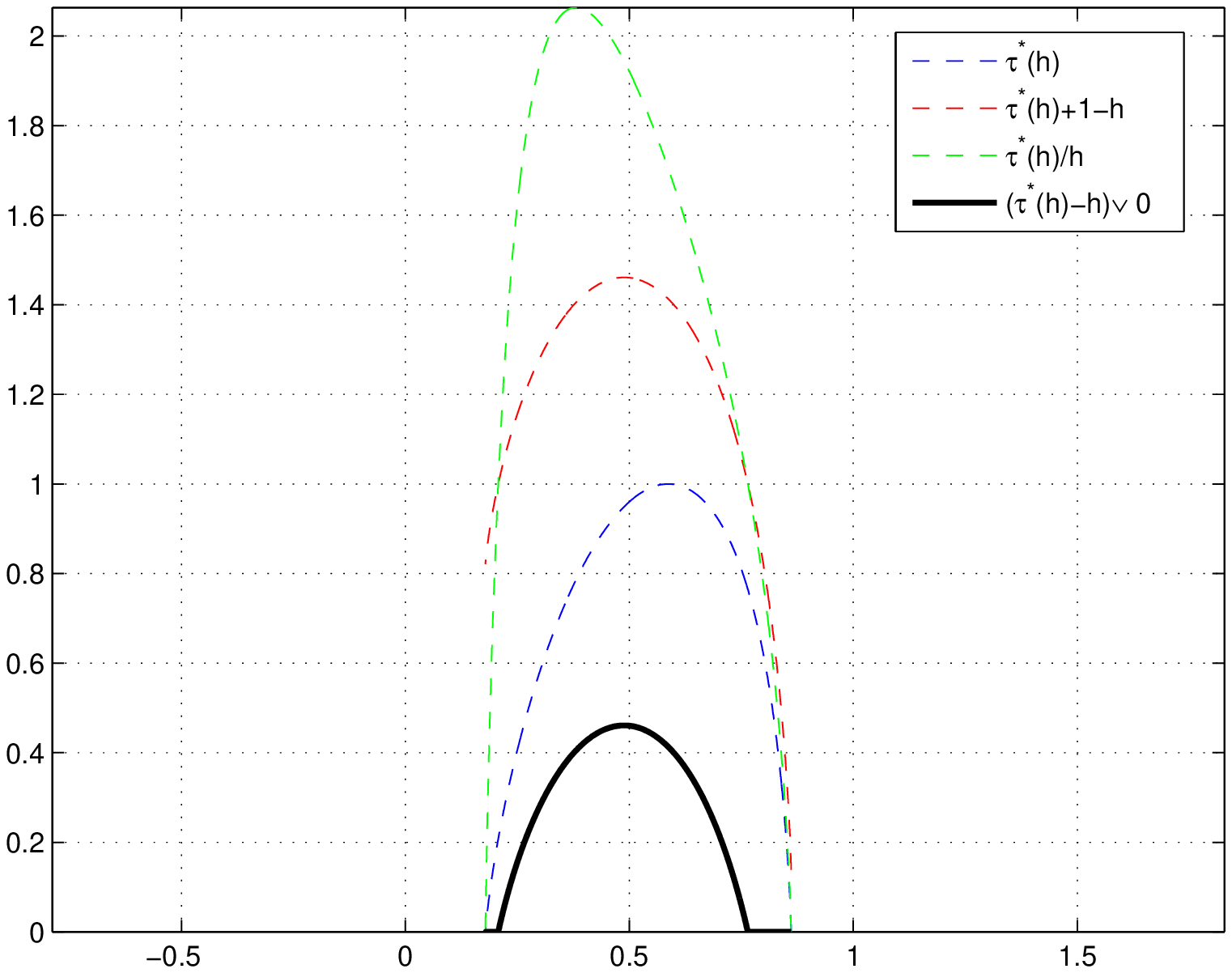} 
\par\vspace{0pt} 
\end{minipage}%
\begin{minipage}[b]{.33\textwidth} 
\centering 
\includegraphics[width=\textwidth]{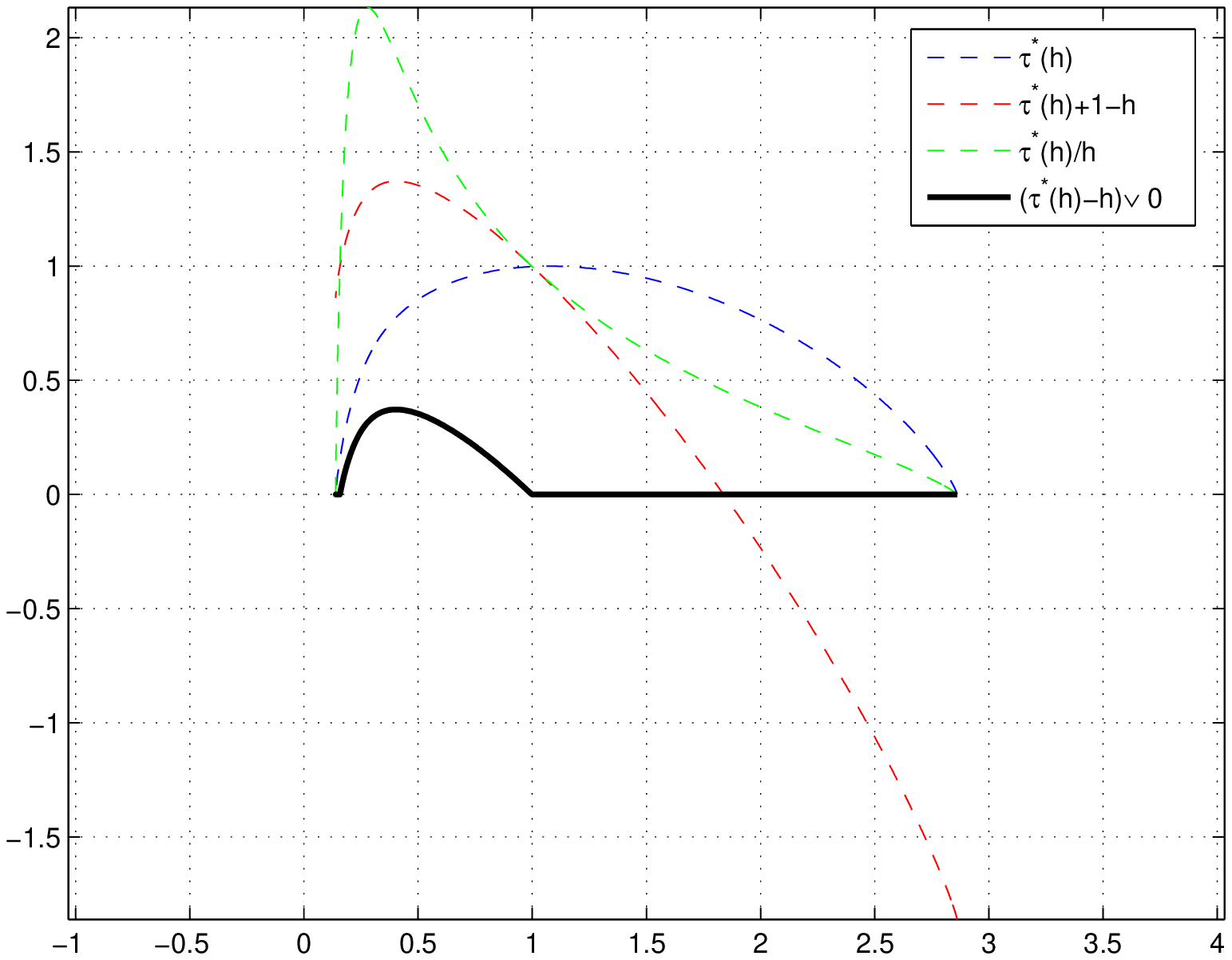} 
\par\vspace{0pt} 
\end{minipage}%
\begin{minipage}[b]{.33\textwidth} 
\centering 
\includegraphics[width=\textwidth]{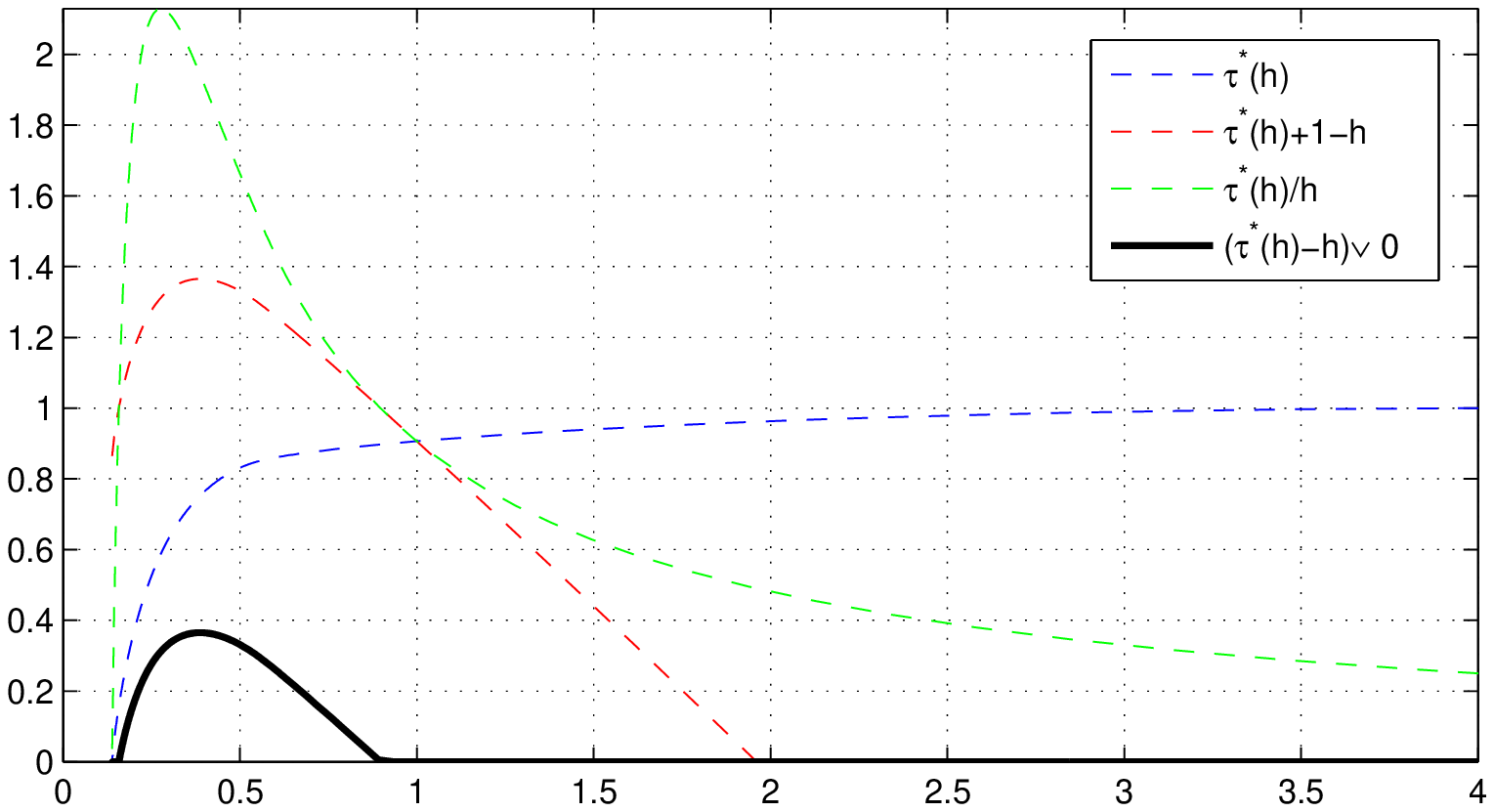} 
\par\vspace{0pt} 
\end{minipage}
\end{center}
\caption{Function $(\tau^*(h)-h)\vee 0$ in case of $\sup J_F < 1$ (Left), $\sup J_F > 1$ (Middle) and $\sup J_F=\infty$ (Right).}
\end{figure}

The rest of the paper is organized as follows: in Section \ref{badic} we introduce the $b$-adic independent cascade function and briefly present its multifractal analysis; in Section \ref{proofmainthm} we prove Theorem \ref{mainthm} with two intermediate results: Theorem \ref{lower} and Theorem \ref{levelsets}, whose proofs are postponed to Section \ref{third}; in Section \ref{second} we prove Theorem \ref{upper} and finally in Section \ref{forth} we prove Proposition \ref{mainprop}, which is our essential tool for proving Theorem \ref{lower} and Theorem \ref{levelsets}.

\section{$b$-adic independent cascade function}\label{badic}

\subsection{Coding space.}

Let $b\ge 2$ be an integer and $\mathscr{A}=\{0,\cdots,b-1\}$ be the alphabet. Let $\mathscr{A}^*=\bigcup_{n\ge 0} \mathscr{A}^n$ (by convention $\mathscr{A}^0=\{\varnothing\}$ the set of empty word) and $\mathscr{A}^{\mathbb{N}_+}=\{0,\dots,b-1\}^{\mathbb{N}_+}$.

Denote the length of $w$ by $|w|=n$ if $w\in\mathscr{A}^n$, $n\ge 0$ and $|w|=\infty$ if $w\in\mathscr{A}^{\mathbb{N}_+}$.

The word obtained by concatenation of $w\in\mathscr{A}^*$ and $t\in\mathscr{A}^*\cup\mathscr{A}^{\mathbb{N}_+}$ is denoted by $w\cdot t$ and sometimes $wt$.

For every $w\in \mathscr{A}^*$, the cylinder with root $w$, i.e. $\{w\cdot t: t\in \mathscr{A}^{\mathbb{N}_+}\}$ is denoted by $[w]$. The set $\mathscr{A}^{\mathbb{N}_+}$ is endowed with the standard metric distance
\[
d(s,t)=\inf\{b^{-n}: n\ge 0,\  \exists \  w\in\mathscr{A}^n,\ s,t\in [w]\}.
\] 

If $n\ge 1$ and $w=w_1\cdots w_n\in \mathscr{A}^n$ then for every $1\le k\le n$, the word $w_1\dots w_k$ is denoted by $w|_k$, and if $k=0$ then $w|_0$ stands for $\varnothing$. Also, for any infinite word $t=t_1 t_2\cdots \in\mathscr{A}^{\mathbb{N}_+}$ and $n\ge 1$, $t|_n$ denotes the word $t_1\cdots t_n$ and $t|_0$ the empty word.

Let
\[
\lambda: t\in \mathscr{A}^*\cup\mathscr{A}^{\mathbb{N}_+} \mapsto \sum_{k=1}^{|t|} t_k\cdot b^{-k}
\]
be the canonical projection from $\mathscr{A}^*\cup\mathscr{A}^{\mathbb{N}_+}$ onto $[0,1]$. For any $x\in[0,1]$ and $n\ge 1$, we define $x|_n=x_1\cdots x_n$ the unique element of $\mathscr{A}^n$ such that $\lambda(x|_n) \le x < \lambda(x|_n)+b^{-n}$ if $t<1$, as well as $1|_n=b-1\cdots b-1$. 

\subsection{$b$-adic independent cascade function}\label{badicF}

Let $(\Omega,\mathcal{A},\mathbb{P})$ be the probability space. Let
\begin{equation*}
W=(W_0,\cdots,W_{b-1}) :\Omega \mapsto \mathbb{R}^b \ \text{ and } \  L=(L_0,\cdots ,L_{b-1}) :\Omega \mapsto (0,1)^b
\end{equation*}
be two random vectors such that $\mathbb{E}(\sum_{j=0}^{b-1} W_j)=\mathbb{E}(\sum_{j=0}^{b-1} L_j)=1$. Let
$$\big\{(W,L)(w)=(W(w),L(w)): w\in \mathscr{A}^* \big\}$$
be a family of independent copies of $(W,L)$.

For any $w\in\mathscr{A}^*$, $u=u_1\cdots u_n\in\mathscr{A}^n$ and $n\ge 1$ we define the products:
\begin{eqnarray}
W_{u}(w)&=& W_{u_1}(w)\cdot W_{u_2}(w\cdot u_1)\cdots W_{u_n}(w\cdot u_1\cdots u_{n-1});\label{Wij}\\
L_{u}(w) &=& L_{u_1}(w)\cdot L_{u_2}(w\cdot u_1)\cdots L_{u_n}(w\cdot u_1\cdots u_{n-1}). \label{Lij}
\end{eqnarray}

For $w \in\mathscr{A}^*$, $n\ge 1$ we define two continuous functions
$$
F_{W,n}^{[w]}(t)=\int_0^t b^{n}\cdot W_{u|_n}(w)\ \mathrm{d} u\ \  \text{ and }\ \  F_{L,n}^{[w]}(t)=\int_0^t b^{n}\cdot L_{u|_n}(w)\ \mathrm{d} u,\ \ t\in[0,1].
$$

For $p\in\mathbb{R}$ and $U\in\{W,L\}$ let
\[
\varphi_U(p)=-\log_b\mathbb{E}(\sum_{j=0}^{b-1}\mathbf{1}_{\{U_j\neq0\}}|U_j|^p).
\]

\smallskip

From~\cite{KP,DL,BMJ} we know that if the following assumption holds:
\begin{itemize}
\item[(A1)] There exists $p>1$ such that $\varphi_{W}(p)>0$, with $p\in(1,2]$ if $\mathbb{P}(\sum_{j=0}^{b-1}W_j=1)<1$. Moreover, $\mathbb{E}(\sum_{j=0}^{b-1} L_j\log L_j)<0$,
\end{itemize}
then for $U\in\{W,L\}$ and for any $w\in\mathscr{A}^*$, $F_{U,n}^{[w]}$ converges uniformly, almost surely and in $L^p$ norm for $p>1$ such that $\varphi_U(p)>0$, as $n$ tends to $\infty$. Moreover, the limit of $F_{L,n}^{w}$ is almost surely increasing.

We denote by $F_W$ and $F_L$ the limits of $F_{W,n}^{[\varnothing]}$ and $F_{L,n}^{[\varnothing]}$. Then, the $b$-adic independent cascade function considered in this paper is
\begin{equation}\label{F}
F=F_W\circ F_L^{-1}: [0,F_L(1)]\mapsto \mathbb{R}.
\end{equation}
We say that we are in the conservative case if $\mathbb{P}(\sum_{j=0}^{b-1} W_j=1)=1$, and in the non-conservative case otherwise.

\begin{remark}
$\ $
\begin{itemize}
\item[(a)] If we set $L=(b^{-1},\cdots, b^{-1})$ and let the entries of $W$ only take positive values then $F$ becomes the indefinite integral of the multiplicative cascades measure $\mu_W$ constructed in~\cite{Mand,KP}. If $W$ is also deterministic, then $\mu_W$ is nothing but the multinomial measure on $[0,1]$ associated with the probability vector $W$.
\smallskip

\item[(b)] If $(W,L)$ is a deterministic pair, then we are in the conservative case and $F$ is the self-affine function studied by Bedford~\cite{Bed1} and Kono~\cite{Kono}, whose multifractal analysis is a consequence of the study of the digit frequency by Besicovitch~\cite{Bes} and Eggleston~\cite{Egg}. The graph and range singularity spectra in this case are still unknown. Even for the dimension of the whole graph, to our best knowledge, there are only results for the box-counting dimension~\cite{Bed1} and in some special cases the Hausdorff dimension~\cite{BeUr1990,FeWa,Ur1990b}. When $W$ is not deterministic but conservative, the situation is very close to that of the deterministic case from the question raised in this paper point of view. Our results will concern the non-conservative case only.
\end{itemize}
\end{remark}

\subsection{Multifractal analysis of $F$.}

The multifractal analysis of $F$ is based on the construction of an uncountable family of statistically self-similar measures $\mu_{q}$ defined on the coding space $\mathscr{A}^{\mathbb{N}_+}$ with desired Hausdorff dimension. More precisely, for $(q,t)\in\mathbb{R}^2$ we define
\begin{equation}\label{Phi}
\Phi(q,t)=\mathbb{E}(\sum_{j=0}^{b-1} \mathbf{1}_{\{W_j\neq 0\}}|W_j|^q\cdot L_j^{-t}).
\end{equation}
Clearly $\Phi(q,t)$ is analytic on the rectangle $\{(q,t)\in\mathbb{R}^2: |\Phi(q,t)|<\infty\}$. Since $L_j\in(0,1)$ for $0\le j\le b-1$, for each $q\in \widetilde{J}:=\{q\in\mathbb{R}: (q,t)\in\mathbb{R}^2, |\Phi(q,t)|<\infty\text{ for some } t\}$ there is a unique $\tau(q)$ such that $\Phi(q,\tau(q))=1$, and the function $\tau$ is easily seen to be concave and analytic over $\widetilde{J}$.

Define the interval $J=\{q\in \widetilde{J}: q\tau'(q)-\tau(q)> 0\}$.

For any $q \in J$ and $w\in\mathscr{A}^*$ define the random vector
$$W_{q}(w)=\left(W_{q, j}(w)= \mathbf{1}_{\{W_j(w)\neq 0\}}|W_j(w)|^q\cdot L_j(w)^{-\tau(q)}\right)_{0\le j\le b-1}.$$

Due to~\eqref{Wij} and~\eqref{Lij}, for $w\in\mathscr{A}^*$, $u\in \mathscr{A}^n$, $n\ge 1$ we can define the product
\begin{equation}\label{Wq}
W_{q,u}(w)= \mathbf{1}_{\{W_u(w)\neq 0\}} |W_{u}(w)|^q\cdot L_{u}(w)^{-\tau(q)}.
\end{equation}

For $q\in J$, $w\in\mathscr{A}^*$ and $n\ge 1$ define
$$Y_{q,n}(w)=\sum_{u \in\mathscr{A}^n} W_{q,u}(w).$$

For $q\in J$, let
\begin{equation}\label{xi(q)}
\xi(q)=-\frac{\partial}{\partial q}\Phi(q,\tau(q))\ \text{ and }\ \widetilde{\xi}(q)=\frac{\partial}{\partial t}\Phi(q,\tau(q)).
\end{equation}
By construction $\tau'(q)=\xi(q)/\widetilde{\xi}(q)$. We present our second assumption:
\begin{itemize}
\item[(A2)] There exists $q<0$ such that $\varphi_W(q)>-\infty$; $\varphi_L$ is finite on $\mathbb{R}$. Moreover, $\mathbb{P}( \sum_{j=0}^{b-1} \mathbf{1}_{\{W_j\neq 0\}} \ge 2 )=1$.
\end{itemize}

Recall that if $\mu$ is a positive Borel measure on a compact metric space, its lower Hausdorff dimensions is defined as ${\dim}_H(\mu)=\inf\{\dim_H E:\mu(E)>0\}$. From~\cite{BJ} we have 

\begin{proposition}\label{aumeas} Suppose that (A1) and (A2) hold.

\medskip

\begin{itemize}
\item[(a)]
With probability 1, for all $q\in J$ and $w\in\mathscr{A}^*$, the sequence $Y_{q,n}(w)$ converges to a positive limit $Y_q(w)$. For any $w\in\mathscr{A}^*\setminus\{\varnothing\}$, the function $J\ni q\mapsto Y_q(w)$ is a copy of $J\ni q\mapsto Y_q(\varnothing):=Y_q$.
\smallskip

\item[(b)]
For every compact subset $K$ of $J$ and $w\in\mathscr{A}^*$ define
\begin{equation}\label{Y_K}
Y_K(w)= \sup_{q\in K}Y_q(w).
\end{equation}
For any $w\in\mathscr{A}^*\setminus\{\varnothing\}$, $Y_K(w)$ is a copy of $Y_K(\varnothing):=Y_K$ and there exists $p_K>1$ such that $\mathbb{E}(Y_K^{p_K})<\infty$.
\smallskip

\item[(c)] With probability 1, for all $q\in J$, the function 
\begin{equation}\label{selfsimimeasure}
\mu_q([w])=W_{q,w}(\varnothing) \cdot Y_q(w),\ w\in\mathscr{A}^* 
\end{equation}
defines a Borel measure on $\mathscr{A}^{\mathbb{N}_+}$ with 
\begin{equation}\label{gammaq}
\dim_H (\mu_q)=\frac{\gamma(q)}{\log(b)}, \text{ where }\gamma(q)=q\xi(q)-\tau(q)\widetilde{\xi}(q).
\end{equation}
 \end{itemize}
 
\end{proposition}

These statistically self-similar measures are the effective tools to study the multifractal behavior of $F=F_W\circ F_L^{-1}$. In fact, with each $\mu_q$ we can induce a measure $\mu^{D}_q$ on the domain $[0,F_L(1)]$: for any Borel set $A\subset\mathbb{R}$,
\begin{equation*}
\mu_q^D(A)=\mu_q\Big(\big\{t \in \mathscr{A}^{\mathbb{N}_+}:F_L\circ \lambda(t)\in A \big\}\Big).
\end{equation*}
It is proved in~\cite{BJ} that, with probability 1, for all $q\in J$, the measure $\mu_q^{D}$ is carried by the set $E_F(\tau'(q))$ and for $\mu_q^{D}$-almost every $x\in E_F(\tau'(q))$, 
$$
h_{\mu_q^D}(x)=\liminf_{r\to0^+}\frac{\log \mu_q^{D}(B(x,r))}{\log r}=q\tau'(q)-\tau(q)=\tau^*(\tau'(q)).$$
Consequently, $\dim_H E_F(\tau'(q))\ge \dim_H(\mu_q^{D})=\tau^*(\tau'(q))$. This is used to obtain the following result in~\cite{BJ}.

\begin{proposition}\label{multiforma}
Suppose that (A1) and (A2) hold.
\medskip

\begin{itemize}
\item[(a)] With probability $1$, $\tau_F=\tau$ on the interval $J$. Moreover, if $\varphi_W$ is finite on $\mathbb{R}$, then if $\overline{q}=\sup J<\infty$ (resp. $\underline{q}:=\inf J>-\infty$) we have $\tau_F(q)=\tau'(\overline{q})q$ (resp. $\tau'(\underline{q})q$) over $[\overline{q},\infty)$ (resp. $(-\infty,\underline{q}]$).
\smallskip

\item[(b)] With probability $1$, $d_{F}=\tau_F^*=\tau^*$ on the interval $\tau'(J)$.
\end{itemize}
\end{proposition}

\section{Proof of Theorem \ref{mainthm}.}\label{proofmainthm}

\subsection{Proof of Theorem \ref{mainthm}(a).}

Since using $\mu_q$ is successful to describe the singularity spectrum of $F$, it is also worth trying to use these measures to study the graph and range singularity spectra of $F$.

\smallskip

For each $q\in J$, like $\mu^D_q$, associated with $F_L$ and $F_W$ we can induce:
\begin{itemize}
\item a measure $\mu^G_q$ carried by the graph: for any Borel set $A\subset\mathbb{R}^2$,
\begin{equation*}
\mu_q^G(A)=\mu_q\Big(\big\{t\in\mathscr{A}^{\mathbb{N}_+}: \big(F_L\circ\lambda(t),F_W\circ\lambda(t)\big)\in A \big\}\Big);
\end{equation*}
\item a measure $\mu^R_q$ carried by the range: for any Borel set $A\subset\mathbb{R}$,
\begin{equation*}
\mu^R_q(A)=\mu_q\Big(\big\{t\in\mathscr{A}^{\mathbb{N}_+}:F_W\circ(t)\in A \big\}\Big).
\end{equation*}
\end{itemize}

We focus on the lower Hausdorff dimension of the measures $\mu^G_q$ and $\mu^R_q$. We show that these dimensions provide the graph and range singularity spectra. Our approach is based on the estimation of the energy of these two measures restricted on suitable random sets (see Remark~\ref{sharpbd}). Before presenting the results, let us introduce our third assumption:

\begin{itemize}
\item[(A3)] $\mathbb{P}(\sum_{j=0}^{b-1}W_j=1)<1$ , $\mathbb{P}(\forall\  j, |W_j|>0)=1$, and $\varphi_W$ is finite on $\mathbb{R}$.
\end{itemize}

\begin{remark}\label{A3}
The first condition in (A3) ensures that $F$ has enough randomness: We have to avoid the conservative case where $\mathbb{P}(\sum_{j=0}^{b-1} W_j=1)=1$. The second condition $\mathbb{P}(\forall\  j, |W_j|>0)=1$ implies that almost surely $F$ is nowhere locally constant. The last condition ensures that the probability distribution of $F_W(1)$ has a bounded density, in fact this condition can be weaken to the existence of a real number $q<-1$ such that $\varphi_W(q)>-\infty$. The existence of the bounded density of $F_W(1)$ is a key property in the proof.
\end{remark}

As an essential intermediate result, we have the following theorem:

\begin{theorem}\label{lower}
Suppose that (A1)-(A3) hold. With probability $1$, for all $q \in J$ we have $\dim_H (\mu_q^G)=\gamma^G(q)$ and $\dim_H (\mu_q^R)=\gamma^{R}(q)$, where
\begin{eqnarray}
\gamma^G(q) &=& \left( \frac{\tau^*(\tau'(q))}{\tau'(q)} \wedge \big(\tau^*(\tau'(q))+1-\tau'(q)\big)\right) \vee \tau^*(\tau'(q)), \label{gammaG} \\
\gamma^{R}(q) &=& \frac{\tau^*(\tau'(q))}{\tau'(q)} \wedge 1. \label{gammaR}
\end{eqnarray}
\end{theorem}

\begin{remark}\label{LedYo}
Notice that for $q\in J$ we have $\dim_H(\mu^D_q)=\tau^*(\tau'(q))$, and we can write
$$
\gamma^G(q)=\tau^*(\tau'(q))+\gamma^R(q)\cdot(1-\tau'(q))\vee 0.
$$
So Theorem~\ref{lower} actually provides us with a Ledrappier-Young like formula~\cite{LeYo} for the uncountable family of statistically self-similar measure $(\mu_q)_{q\in J}$ uniformly: with probability $1$, for all $q\in J$,
$$
\dim_H (\mu_q^G) =\dim_H (\mu_{q}^D)+\dim_H(\mu_{q}^R)\cdot (1-\tau'(q)) \vee 0.
$$
Similar formula also appear in Theorem 12 in~\cite{BeUr1990}, Theorem 3 in~\cite{FeWa} and Corollary 5.2 in~\cite{Baran1} in the study of the Hausdorff dimension of self-affine measures and sets.
\end{remark}

\begin{remark}\label{sharpbd}
It is worth noting that to prove Theorem~\ref{lower} we are forced to calculate the energy of $\mu_q$ restricted to suitable Cantor-like random sets (see Section~\ref{cantor}). If we did not use this restriction, for example for the measure $\mu_q^G$ in the case where its dimension is greater than $1$, we would have to estimate the expectation of
$$
\iint_{s,t\in \mathscr{A}^{\mathbb{N}_+}} \frac{\mathrm{d}\mu_q(s)\mathrm{d}\mu_q(t)}{(|F_L\circ\lambda(s)-F_L\circ\lambda(t)|^2+|F_W\circ\lambda(s)-F_W\circ\lambda(t)|^2)^{\gamma/2}},\ \gamma>1, 
$$
which turns out to be finite only if $\Phi(2q-1,\gamma-1+2\tau(q))<1$, which is equivalent to saying that $\gamma<1+\tau(2q-1)-2\tau(q)$. So the best lower bound we would get is:
\begin{equation}\label{sharp}
\dim_H (\mu^G_q) \ge 1+\tau(2q-1)-2\tau(q).
\end{equation}
Comparing this value with the exact dimension $1+(q-1)\tau'(q)-\tau(q)$, we find that~(\ref{sharp}) always provides a strick lower bound unless $q=1$. Thus, such an approach only provides the Hausdorff dimension of the whole graph.
\end{remark}

Since $\mu_q^{D}$ is carried by the set $E_F(\tau'(q))$, by definition the measure $\mu_q^{S}$ is carried by the set $S_F(E_F(\tau'(q)))$ for $S\in\{G,R\}$. Then combining the results in Proposition \ref{aumeas}, Proposition \ref{multiforma} and Theorem \ref{lower}, we prove the results on the singularity spectra part of Theorem \ref{mainthm}(a).

\medskip

For the result on the dimension of the whole graph, let $I=[0,F_L(1)]$. For each $n\ge 1$ we divide $I$ into $b^n$ semi-open to the right intervals of the same length denoted by $I_{n,k}$, for $k=1,\cdots, b^{n}$.

Recall that $\mathrm{Osc}_F(I_{n,k})=\sup_{x,y\in I_{n,k}} |F(x)-F(y)|$, so for each interval $I_{n,k}$ we will need at most $[\frac{\mathrm{Osc}_F(I_{n,k})}{|I_{n,k}|}]+1$ many squares whose side length is $|I_{n,k}|$ to cover $G_F(I_{n,k})$. Then, by definition of the upper box-counting dimension and the definition of $\tau_F$ in~\eqref{tauq} we get
$$
\overline{\dim}_B G_F\le \limsup_{n\to\infty}\frac{\log \sum_{j=1}^{b^n} ([\frac{\mathrm{Osc}_F(I_{n,k})}{|I_{n,k}|}]+1)}{ -\log (b^{-n}\cdot F_L(1))}\le 1+(-\tau_F(1))\vee 0.
$$
From Proposition~\ref{multiforma} we know that almost surely $\tau_F(1)=\tau(1)\le 0$ and applying Theorem \ref{lower} to $q=1$ we get with probability 1,
$$\dim_H  G_F\ge \dim_H G_F(\tau_F'(1)) =1-\tau_F(1).$$
Consequently, with probability 1,
$$\dim_H  G_F=\dim_P  G_F=\dim_B  G_F=1-\tau_F(1).$$

\begin{remark}
If we consider the exponent $\tilde{h}_F(x)$ defined in~\eqref{limex} and consider the following smaller iso-H\"older sets
$$
\widetilde{E}_f(h):=\{x\in I: \tilde{h}_f(x)=h\},
$$
then we claim that the results in Theorem~\ref{mainthm}(a) also works for the packing dimension if we replace the mono-H\"older set $E_f(h)$ by the set $\widetilde{E}_f(h)$, since in this case we have $\dim_P \widetilde{E}_f(h) \le \tau_f^*(h)$ for $h> 0$ and, moreover, under the assumption $\mathbb{P}(\forall j, |W_j|>0)=1$ in (A3), the measure $\mu_q^{D}$ is actually carried by the set $\widetilde{E}_F(\tau'(q))$ for $q\in J$ (see~\cite{BJ}).
\end{remark}

\subsection{Proof of Theorem \ref{mainthm}(b).}

To get measures on the level sets, in the same spirit as when one constructs the local times of certain stochastic processes, we could disintegrate the measures $\mu_q^G$ with respect to $\mu_q^R$ in order to obtain Radon measures $\mu_q^y$ carried by $L^y_F$ for $\mu_q^R$-almost every $y$, but such a disintegration turns out to be difficult to study. The reason is that the energy method we use does not provide the exact gauge function needed to describe the density of the measure $\mu_q^{R}$ with respect to Lebesgue. It only yields the lower Hausdorff dimension of $\mu_q^{R}$. However, inspired by what is done in~\cite{Orey1971} to calculate the Hausdorff dimension of the level sets of Gaussian process by using classical Martrand theorem, and in~\cite{Mat} to deal with the the Hausdorff dimension of slices of sets, it is possible to solve this problem for for Lebesgue almost every direction.

\medskip

For $h> 0$ and $\theta\in(-\pi/2,\pi/2)$, recall that $R_{F,\theta}(h)=\mathrm{Proj}_\theta(G_F(h))$, and for each $y\in R_{F,\theta}(h)$ recall that $L^y_{F,\theta}(h)=G_F(h)\cap l_{y,\theta}^{\perp}$.

For $q\in J$, let $\mu_{q,\theta}^{R}$ be the orthogonal projection of the measure $\mu_q^{G}$ onto $l_{\theta}$: $\mu_{q,\theta}^{R}(A) = \mu_q^{G}\circ \mathrm{Proj}_\theta^{-1} (A)$ for any Borel set $A\subset l_{\theta}$. Since $\mu_q^G$ is carried by $G_F(\tau'(q))$, so $\mu_{q,\theta}^{R}$ is carried by $R_{F,\theta}(\tau'(q))$.

\medskip

We have the following theorem:

\begin{theorem}\label{levelsets}
Suppose that (A1)-(A3) hold. With probability 1, for Lebesgue almost every $\theta\in(-\pi/2,\pi/2)$, for all $q\in J$ such that $\dim_H(\mu_q^G)=\gamma^{G}(q)>1$:

\medskip

\begin{itemize}
\item[(a)] The projected measure $\mu_{q,\theta}^{R}$ is absolutely continuous with respect to the one-dimensional Lebesgue measure on the $l_{\theta}$.
\smallskip

\item[(b)] For $\mu_{q,\theta}^{R}$-almost every $y\in l_{\theta}$, the following limit:
\begin{equation*}
\lim_{r\to 0^+}\frac{1}{r} \int_{x\in\mathbb{R}^2, |x-l_{y,\theta}^{\perp} |\le r} \psi(x) \ \mathrm{d} \mu_q^{G}(x)
\end{equation*}
exists for any continuous function $\psi:\mathbb{R}^2\mapsto \mathbb{R}_+$, so it defines a measure $\mu^{y}_{q,\theta}$ carried by $L^y_{F,\theta}(\tau'(q))$. 
\smallskip

\item[(c)] There exists a random set $R_{q,\theta}\subset R_{F,\theta}(\tau'(q))$ of full $\mu_{q,\theta}^{R}$-measure such that for any $y\in R_{q,\theta}$, the measure $\mu^{y}_{q,\theta}$ has lower Hausdorff dimension
$$\dim_H (\mu^{y}_{q,\theta})= \dim_H (\mu^{G}_{q})-1=\tau^*(\tau'(q))-\tau'(q).$$
\end{itemize}
\end{theorem}

Theorem \ref{mainthm}(b) is almost a direct consequence of Theorem \ref{upper}(b) and Theorem \ref{levelsets}, we only remark that $\gamma^G(q)>1$ if and only if $\tau'(q)<1$ and $\tau^*(\tau'(q))-\tau'(q)>0$.

\begin{remark}\label{endpoints}
As mentioned in Remark \ref{A3}, the condition ``$\varphi_W$ is finite on $\mathbb{R}$" can be weakened to ``there exists $q<-1$ such that $\varphi_W(q)$ is finite". Under this weaker assumption, the results in Theorem \ref{mainthm} will still hold, but only for $h\in \{h>0: \tau^*(h)>0\}$. The reason why we cannot conclude for $\{h>0:\tau_F^*(h)>0\}$ is that under this weaker assumption we do not know the value of $\tau_F$ outside the interval $J=\{q:q\tau'(q)-\tau(q)>0\}$. But if we assume ``there exist $\underline{q}<\overline{q}\in J$ such that $\underline{q}\tau'(\underline{q})-\tau(\underline{q})=\overline{q}\tau'(\overline{q})-\tau(\overline{q})=0$", then we will also obtain the same results as in Theorem \ref{mainthm}.
\end{remark}

\section{Proof of Theorem~\ref{upper}.}\label{second}

\subsection{Results on the packing dimension.}

\begin{proof}
Without loss of generality we suppose $I=[0,1]$. For any $x\in[0,1]$ and $n\ge 1$, we denote by $I_{x,n}$ the unique interval $[k\cdot 2^{-n},(k+1)\cdot 2^{-n})$, $k\in\mathbb{Z}$ that contains $x$.
 
Let $\{E_i\subset E,i\ge 1\}$ be any countable covering of $E$. The families $\{R_i=R_f(E_i)\subset R_f(E), i\ge 1\}$ and $\{G_i=G_f(E_i)\subset G_f(E, i\ge 1\}$ are countable coverings of $R_f(E)$ and $G_f(E)$ respectively, so 
$$\dim_{P} R_f(E) = \sup_{i} \dim_{P} R_i,\ \   \dim_{P} G_f(E) = \sup_{i} \dim_{P} G_i.
$$
For any $i\ge 1$, $k\ge 0$ and $\epsilon>0$ with $h-\epsilon>0$ we define
$$E_{i,k}=\{x\in E_i : \forall\ n\ge k, O_f(I_{x,n}) \le |I_{x,n}|^{h-\epsilon}\},
$$
as well as $R_{i,k}=R_f(E_{i,k})$ and $G_{i,k}=G_f(E_{i,k})$. By the fact that for any $x\in E_i\subset E$,
$$\liminf_{n\to\infty} \frac{\log O_f(I_{x,n})}{\log |I_{x,n}|} \ge h_f(x) \ge h,$$
we have $\bigcup_{k\ge 0}E_{i,k} = E_i$, thus $\bigcup_{k\ge 0} R_{i,k}=R_i$ and $\bigcup_{k\ge 0} G_{i,k}=G_i$.

\smallskip

For $n\ge 1$ and $A\subset \mathbb{R}^2$, let $N_n(A)$ be the minimal number of dyadic squares of the form $[k\cdot 2^{-n},(k+1)\cdot 2^n) \times[l\cdot 2^{-n},(l+1)\cdot 2^{-n}), \ 0\le k<2^n-1,\ l\in\mathbb{Z}$ necessary to cover $A$. By the construction of $E_{i,k}$, for any $n\ge k$ we have 
\begin{eqnarray*}
N_n(R_{i,k})\le 2\cdot N_{[\frac{n}{h-\epsilon}]+1}(E_{i,k}),\ 
N_n(G_{i,k}) \le\left\{
\begin{array}{l}
2^{n\cdot ((1-h+\epsilon)\vee 0)}\cdot N_n(E_{i,k}); \\
2\cdot N_{[\frac{n}{(h-\epsilon)\wedge1}]+1}(E_{i,k}).
\end{array}
\right.
\end{eqnarray*}
This gives us 
\begin{eqnarray*}
\overline{\dim}_{B} R_{i,k} \le \frac{1}{h-\epsilon}\overline{\dim}_{B} E_{i,k},\ 
\overline{\dim}_{B} G_{i,k} \le\left\{
\begin{array}{l}
(1-h+\epsilon) \vee 0+\overline{\dim}_{B} E_{i,k}; \\
\overline{\dim}_{B} E_{i,k}/((h-\epsilon)\wedge 1).
\end{array}
\right.
\end{eqnarray*}

We know that the packing dimension is equal to the upper modified box-counting dimension (see Chapter 3,~\cite{Falc}), that is, for any set $A\subset\mathbb{R}^2$,
$$\dim_P A=\overline{\dim}_{MB} A:=\inf \left\{\sup_{i} \overline{\dim}_{B} A_i: A\subset\bigcup_{i} A_i \right\},$$
where the infimum is taken over all possible countable coverings of $A$. Since $\{R_{i,k},k\ge 0\}$ and $\{G_{i,k},k\ge 0\}$ are countable coverings of $R_i$ and $G_i$, we have
\begin{eqnarray*}
\dim_{P} R_i &=& \overline{\dim}_{MB} R_i \le \sup_{k} \overline{\dim}_{B} R_{i,k} \le \frac{1}{h-\epsilon} \sup_{k} \overline{\dim}_{B} E_{i,k};\\
\dim_{P} G_i &=& \overline{\dim}_{MB} G_i \le \sup_{k} \overline{\dim}_{B} G_{i,k} \le
\left\{
\begin{array}{l} 
(1-h+\epsilon) \vee 0 + \sup_{k}\overline{\dim}_{\mathbf{B}} E_{i,k};\\
\sup_{k}\overline{\dim}_{B} E_{i,k}/((h-\epsilon)\wedge 1).
\end{array}
\right.
\end{eqnarray*}
Since for any $k\ge 0$ we have $E_{i,k}\subset E_{i}$ so that $\sup_{k}\overline{\dim}_{B} E_{i,k} \le \overline{\dim}_{B} E_{i}$. Then we have shown that for any countable covering $\{E_i,i\ge 1\}$,
\begin{eqnarray*}
\dim_{P} R_f(E) &=&\sup_{i} \dim_{P} R_i \le \frac{1}{h-\epsilon} \sup_{i} \overline{\dim}_{B} E_{i};\\
\dim_{P} G_f(E) &=&\sup_{i} \dim_{P} G_i \le 
\left\{
\begin{array}{l} 
(1-h+\epsilon)\vee 0+\sup_{i} \overline{\dim}_{B} E_{i};\\
\sup_{i}\overline{\dim}_{B} E_{i}/((h-\epsilon)\wedge1).
\end{array}
\right.
\end{eqnarray*}
By taking the infimum over all the possible $\{E_i,i\ge 1\}$ we get
\begin{eqnarray*}
\dim_{P} R_f(E) \le \frac{1}{h-\epsilon} \dim_P E,\ \dim_{P} G_f(E)  \le
\left\{
\begin{array}{l} 
(1-h+\epsilon)\vee 0+\dim_{P} E;\\
\dim_{P} E/((h-\epsilon)\wedge1).
\end{array}
\right.
\end{eqnarray*}
Letting $\epsilon$ tend to $0$ yields the conclusion. \end{proof}

\subsection{Results on the Hausdorff dimension.}

\begin{proof}

We fix $\epsilon\in (0,h)$ and $\theta\in(-\pi/2,\pi/2)$. Recall that $\mathrm{Proj}_\theta$ is the orthogonal projection of $\mathbb{R}^2$ onto $l_\theta$ and $R_{f,\theta}(E)=\mathrm{Proj}_\theta(G_f(E))$.

For $x\in \mathbb{R}$, $r>0$ and $k\in\mathbb{Z}$ we define the interval $I_k(x,r)=[x+(2k-1)r,x+(2k+1)r]$ and the square $Q_k(x,r)=[x-r,x+r]\times I_k(f(x),r)$.

For any $r>0$ let $n(r)=[r^{h-\epsilon-1}]+1$ and $\mathcal{N}(r)=\{k\in\mathbb{Z}: 2|k|\le [r^{h-\epsilon-1}]\}$. We have $\max\{r,r^{h-\epsilon}\}\le n(r)\cdot r$ and $\# \mathcal{N}(r)\le n(r)$.

For any $x\in E$ and $r>0$ define the family $\mathcal{Q}(x,r)=\Big\{Q_k(x,r):k\in \mathcal{N}(r)\Big\}$ and for any $y\in R_{f,\theta}(E)$ define the family
\begin{equation*}
\mathcal{Q}^y_\theta(x,r)=\Big\{Q_k(x,r): k\in\mathcal{N}(r), \ Q_k(x,r)\cap l_{y,\theta}^{\perp}\neq \emptyset \Big\}.
\end{equation*}
By simple calculation we know $\# \mathcal{Q}^y_\theta(x,r) \le |\tan \theta|+1$.

Let $d$ stand for $\dim_{H} E$. By definition of the Hausdorff dimension we can find a decreasing sequence $(\delta_i)_{i\ge 1}$ tending to $0$ and for each $i\ge 1$,
$$\mathscr{B}_i:=\Big\{B_j^{(i)}=B\left(x_j^{(i)},r_j^{(i)}\right)\Big \}_{j\in\mathcal{J}_i},$$
a countable $\delta_i$-covering of $E$ such that $ x_j^{(i)}\in E$ for $j\in  \mathcal{J}_i$, and $\sum_{j\in\mathcal{J}_i} (r^{(i)}_j)^{d+\epsilon} \leq 2^{-i}$.

For $x\in E$ and $r>0$ denote the rectangle:
\[
R(x,r)=B(x,r) \times \big[f(x)-n(r)\cdot r,f(x)+n(r)\cdot r\big],
\]
and for any $B_j^{(i)}=B\left(x_j^{(i)},r_j^{(i)}\right)\in \mathscr{B}_i$, we use the convention $R^{(i)}_j=R\left(x_j^{(i)},r_j^{(i)}\right)$.

Let $\mu$ be a positive Borel measure defined on $l_\theta$. For $\gamma>0$ recall that
$$R_{f,\theta}^{\mu,\gamma}(E)=\{y\in R_{f,\theta}(E): h_\mu(y) \ge \gamma\}.$$
Suppose that $\mu(R_{f,\theta}^{\mu,\gamma}(E))>0$. Then define a subset of $\mathcal{J}_i$:
\begin{equation*}
\mathcal{J}^{\mu,\gamma}_{i,\theta}=\Big\{j\in \mathcal{J}_i: \mu\Big(\mathrm{Proj}_\theta\big(R^{(i)}_j\big)\Big) \le \Big(\big(2n(r^{(i)}_j)\cdot\cos\theta+|\sin\theta|\big) \cdot r^{(i)}_j\Big)^{\gamma-\epsilon} \Big\}.
\end{equation*}
We have the following lemma:

\begin{lemma}\label{covering}
For any $N\geq 1$, let $\mathscr{C}^{R}_{N}=\bigcup_{i\geq N} \bigcup_{j \in \mathcal{J}_i} \left\{I_0(f(x^{(i)}_j), (r_j^{(i)})^{h-\epsilon})\right\}$,
$$
\mathscr{C}^{G}_{N}=\bigcup_{i\geq N} \bigcup_{j\in \mathcal{J}_i} \mathcal{Q}(x^{(i)}_j,r_j^{(i)}),\ \widetilde{\mathscr{C}^G_N}=\bigcup_{i\geq N} \bigcup_{j \in \mathcal{J}_i} \left\{Q_0(x^{(i)}_j, n(r_j^{(i)})\cdot r_j^{(i)})\right\},
$$
and for any $y\in R_{f,\theta}^{\mu,\gamma}(E)$ let $\mathscr{C}^{y}_{\theta,N}=
\bigcup_{i\geq N} \bigcup_{j\in \mathcal{J}^{\mu,\gamma}_{i,\theta}} \mathcal{Q}^y_\theta(x^{(i)}_j,r_j^{(i)})$.

Then $\mathscr{C}^{R}_{N}$, $\mathscr{C}^{G}_{N}$, $\widetilde{\mathscr{C}^G_N}$ and $\mathscr{C}^{y}_{\theta,N}$ form respectively a $(\delta_N)^{h-\epsilon}$-covering of $R_f(E)$, a $\delta_N$-covering of $G_f (E)$, an $n(\delta_N)\cdot \delta_N$-covering of $G_f (E)$, and a $\delta_N$-covering of $L_{f,\theta}^{y}(E)$.

\end{lemma}

\begin{proof}  Fix $N\ge 1$. For any $x\in E$, since for any $i\ge N$ there are balls in $\mathscr{B}_i$ covering $x$ and $\delta_i\searrow 0$, we can find a sequence of balls $\{B_l=B(x_l,r_l)\}_{l\ge 1}\subset \bigcup_{i\geq N}\mathscr{B}_i$ such that $x\in B_l$ for all $l\ge 1$ and $r_l\searrow 0$ as $l\to\infty$. For each $l\ge 1$, let $\bar r_l=|x-x_l|\le r_l$ and $\bar r'_l=(2n(r_l)\cos\theta+|\sin\theta|)\cdot r_l$. Since
\[
\liminf_{l\to \infty} \frac{\log O_f(B(x,\bar r_l))}{\log \bar r_l}\ge h_f(x)\ge h \ \text{ and}
\]
\[
\liminf_{l\to \infty} \frac{\log \mu\Big(l_\theta\cap B\big(\mathrm{Proj}_\theta(f(x)),\bar r'_l\big)\Big)}{\log \bar r'_l}\ge h_\mu\big(\mathrm{Proj}_\theta(f(x))\big) \ge \gamma,
\]
if $\mathrm{Proj}_\theta(f(x)) \in R_{f,\theta}^{\mu,\gamma}(E)$, so we can find $l_*$ (depending on $x$) such that for all $l\ge l_*$,
$$O_f(B(x,\bar r_{l}))\le (\bar r_{l})^{h-\epsilon} \ \text{ and }\  \mu\Big(l_\theta\cap B\big(\mathrm{Proj}_\theta(f(x)),\bar r'_l\big)\Big) \le (\bar r'_l)^{\gamma-\epsilon}.$$
Now for any $l\ge l_*$ we have
$$|f(x_{l})-f(x)|\le O_f(B(x,\bar x_{l})) \le  (\bar r_{l})^{h-\epsilon} \le (r_{l})^{h-\epsilon}\leq n(r_{l})\cdot r_{l}.$$
This implies that:
\begin{itemize}
\item $x\in B(x_l,r_l)$;
\item $f(x)\in I_0(f(x_l),(r_l)^{h-\epsilon})\subset I_0(f(x_l),n(r_l)r_l)$;
\item $(x,f(x)) \in Q_0(x_l, n(r_l)\cdot r_l)$;
\item $l_\theta\cap B\big(\mathrm{Proj}_\theta(f(x)),\bar r'_l\big)\supset \mathrm{Proj}_\theta (R(x_l,r_l))$,
\end{itemize}
which gives us the conclusion. \end{proof}

Now we are going to show that the coverings constructed in Lemma~\ref{covering} lead to the expected upper bounds. In order to simplify the proof, we use the convention $|Q|=\frac{1}{2}\sup_{x,y\in Q}|x-y|$, the half-diameter of the set $Q$.

\begin{itemize}
\item[(i)] Since we took $\epsilon\in (0,h)$ we have
$$\sum_{Q\in \mathscr{C}^{R}_{N}} |Q|^{\frac{d+\epsilon}{h-\epsilon}} = \sum_{i\ge N} \sum_{j\in \mathcal{J}_i} ((r_j^{(i)})^{h-\epsilon})^{\frac{d+\epsilon}{h-\epsilon}}=  \sum_{i\ge N} \sum_{j\in\mathcal{J}_i} (r^{(i)}_j)^{d+\epsilon} \le 2^{-N+1}.$$
\item[(ii)] If $h> 1$, if we take $\epsilon$ small enough so that $h-\epsilon>1$, then $n(r)=1$ for all $r<1$, and for $N$ large enough so that $\delta_N<1$,
$$\sum_{Q\in \widetilde{\mathscr{C}^G_N}} |Q|^{d+\epsilon} = \sum_{i\ge N} \sum_{j\in \mathcal{J}_i} (n(r_j^{(i)})\cdot r_j^{(i)})^{d+\epsilon}=  \sum_{i\ge N} \sum_{j\in\mathcal{J}_i} (r^{(i)}_j)^{d+\epsilon} \le 2^{-N+1}.$$
\item[(iii)] If $h \le 1$, then $h-\epsilon-1\le -\epsilon<0$, thus for $r<1$ we have $r^{h-\epsilon-1}>1$, and this implies that $n(r)\le 2\cdot r^{h-\epsilon-1}$, hence for $N$ large enough so that $\delta_N<1$,
$$\sum_{Q\in \widetilde{\mathscr{C}^G_N}} |Q|^{\frac{d+\epsilon}{h-\epsilon}} = \sum_{i\ge N} \sum_{j \in\mathcal{J}_i} (n(r_j^{(i )})\cdot r_j^{(i )})^{\frac{d+\epsilon}{h-\epsilon}} \le 2^{\frac{d+\epsilon}{h-\epsilon}} \cdot \sum_{i\ge N} \sum_{j \in\mathcal{J}_i} (r_j^{(i )})^{d+\epsilon}\le 2^{\frac{d+\epsilon}{h-\epsilon}} \cdot 2^{-N+1}.$$
\item[(iv)] If $h \le 1$, for the same reason as (c), for $N$ large enough such that $\delta_N<1$,
\begin{eqnarray*}
&&\sum_{Q\in \mathscr{C}^{G}_{N}} |Q|^{d+1-h+2\epsilon} =\sum_{i\ge N} \sum_{j\in \mathcal{J}_i}\sum_{k\in \mathcal{N}(r^{(i)}_j)} |Q_k(x^{(i)}_j,r_j^{(i)})|^{d+1-h+2\epsilon}\\
&\le& \sum_{i\ge N} \sum_{j\in\mathcal{J}_i} n(r^{i}_j)\cdot (r_j^{(i )})^{d+1-h+2\epsilon} \le 2 \cdot \sum_{i\ge N} \sum_{j \in\mathcal{J}_i} (r_j^{(i )})^{d+\epsilon}\le 2^{-N+2}.
\end{eqnarray*}
\end{itemize}

Now by letting $N$ tend to infinity and then $\epsilon$ to $0$ we obtain the desired upper bound for $\dim_H R_f(E)$ and $\dim_H G_f(E)$. 

Next we prove the upper bound for the Hausdorff dimension of the level sets.

If $s=d+\epsilon-(\gamma-\epsilon)(h-\epsilon)>0$ and $h\le 1$, then for any $N$ large enough we have

\begin{eqnarray*}
&&\int_{y\in R_{f,\theta}^{\mu,\gamma}(E)} \sum_{Q\in \mathscr{C}^y_{\theta,N}} |Q|^{s} \mathrm{d}\mu(y)\\
&\le&\sum_{i\ge N} \sum_{j\in \mathcal{J}^{\mu,\gamma}_{i,\theta}} \sum_{Q \in \mathcal{Q}^y_\theta(x^{(i)}_j,r^{(i)}_j)} |Q|^{s}\cdot \mu\Big(\mathrm{Proj}_\theta\big(Q\big)\Big) \\
&\le&\sum_{i\ge N} \sum_{j\in \mathcal{J}^{\mu,\gamma}_{i,\theta}} (r_j^{(i)})^{s}\cdot (|\tan\theta|+1) \cdot \mu\Big(\mathrm{Proj}_\theta\big(R(x^{(i)}_j,r^{(i)}_j)\big)\Big) \\
&&\qquad \qquad \qquad \qquad\qquad  \Big( \text{since } |Q|=r^{(i)}_j \text{ and }\# \mathcal{Q}^y_\theta(x^{(i)}_j,r^{(i)}_j)\le |\tan\theta|+1 \Big)\\
&\le& (|\tan\theta|+1) \sum_{i\ge N} \sum_{j\in \mathcal{J}^{\mu,\gamma}_{i,\theta}} (r_j^{(i)})^{s}\cdot \Big(\big(2n(r^{(i)}_j)\cdot\cos\theta+|\sin\theta|\big) \cdot r^{(i)}_j\Big)^{\gamma-\epsilon} \\
&\le& C \sum_{i\ge N} \sum_{j\in \mathcal{J}^{\mu,\gamma}_{i,\theta}} (r_j^{(i)})^{s}\cdot \big((r^{i}_j)^{h-\epsilon-1}\cdot r^{(i)}_j\big)^{\gamma-\epsilon} \ \Big(\text{we used } h\le 1 \text{ and } N \text{ large enough} \Big)\\
&=& C \sum_{i\ge N} \sum_{j \in \mathcal{J}^{\mu,\gamma}_{i,\theta}} (r_j^{(i)})^{d+\epsilon} \le C \sum_{i\ge N} \sum_{j \in \mathcal{J}_i} (r_j^{(i)})^{d+\epsilon}  \le 2^{-N+1+2\gamma},
\end{eqnarray*}
where $C=2^{2(\gamma-\epsilon)}(|\tan\theta|+1)$. Due to Borel-Cantelli lemma we get for $\mu$-almost every $y\in R_{f,\theta}^{\mu,\gamma}(E)$,
$$\dim_H L^y_{f,\theta}(E)\le s=d+\epsilon-(\gamma-\epsilon)(h-\epsilon).$$
Applying this with a sequence $(\epsilon_n)_{n\ge 1}\searrow 0$ we get the conclusion. \end{proof}

\section{Proof of Theorem~\ref{lower} and Theorem~\ref{levelsets}.}\label{third}

From now on we assume that (A1)-(A3) hold.

\subsection{Cantor-like subsets of $\mathscr{A}^{\mathbb{N}_+}$ carrying $\mu_q$.} \label{cantor}

For any $w\in\mathscr{A}^*$, define the $b$-adic interval $I_w=\lambda([w])$.  By construction we know that the limit functions $F_W$ and $F_L$ satisfy the following functional equation: For any $w\in \mathscr{A}^*$, $x,y\in I_w$ and $U\in\{W,L\}$,
\begin{equation}\label{WL}
F_U(x)-F_U(y)=U_{w}(\varnothing)\cdot \left(F_U^{[w]}(b^{|w|}\cdot (x-\lambda(w))-F_U^{[w]}(b^{|w|}\cdot (y-\lambda(w))\right).
\end{equation}
For $w\in\mathscr{A}^*$ and $U\in\{W,L\}$ we define the oscillations $O_U(w)=O_{F_U^{[w]}}([0,1])$. Then from~(\ref{WL}) we get that for any $w \in \mathscr{A}^*$ and $U\in\{W,L\}$,
\begin{equation}\label{osc1}
O_{F_U}(I_w)=U_{w}(\varnothing)\cdot O_U(w).
\end{equation}

For $w\in\mathscr{A}^*$ denote by $w^-$ (resp.  $w^+$)  the unique element of $\mathscr{A}^{|w|}$ such that $\lambda(w^-)=\lambda(w)-b^{-|w|}$ (resp. $\lambda(w^+)=\lambda(w)+b^{-|w|}$)  whenever $\lambda(w)\neq 0$ (resp. $\lambda(w)\neq 1-b^{-|w|}$).

Recall~\eqref{xi(q)}, the definition of $\xi(q)$ and $\widetilde{\xi}(q)$. For any $q\in J$, $\epsilon>0$, $u,v\in\mathscr{A}^*$ we define the following subsets of $\Omega$:
\begin{equation}\label{indicLWO}
{\small
\left\{
\begin{array}{l} 
\mathscr{W}_{v}^{[u]}(q,\epsilon) = \left\{ \omega\in\Omega\ :\ \Big\{W_{v}(u),\ W_{v^-}(u),\ W_{v^+}(u)\Big\} \subset [e^{-|v|(\xi(q)+\epsilon)},\ e^{-|v|(\xi(q)-\epsilon)}]\right\};  \\
\mathscr{L}_{v}^{[u]}(q,\epsilon) = \left\{ \omega\in\Omega\ :\ \Big\{L_{v}(u),\ L_{v^-}(u),\ L_{v^+}(u)\Big\}\subset [e^{-|v|(\widetilde{\xi}(q)+\epsilon)},\ e^{-|v|(\widetilde{\xi}(q)-\epsilon)}]\right\};\\
\mathscr{O}_{v}^{[u]}(\epsilon) = \left\{\omega\in\Omega\ :\ \Big\{O_U(uv),\ O_U(uv^-),\ O_U(uv^+): U\in\{W,L\}\Big\} \subset [e^{-|v|\epsilon},\ e^{|v|\epsilon}] \right\}.
\end{array}
\right.
}
\end{equation}

For $q\in J$, $\epsilon>0$, $u,v\in\mathscr{A}^*$ we define the indicator function:
\begin{equation}\label{indicator}
\mathbf{1}_{v}^{[u]}(q,\epsilon)=\mathbf{1}_{ \mathscr{W}_{v}^{[u]}(q,\epsilon)\cap \mathscr{L}_{v}^{[u]}(q,\epsilon)\cap\mathscr{O}_{v}^{[u]}(\epsilon)}.
\end{equation}
and for $q\in J$ and $\epsilon>0$ we define random subsets of $\mathscr{A}^{\mathbb{N}_+}$: 
$$\mathscr{A}^{\mathbb{N}_+}_n(q,\epsilon)=\{t\in\mathscr{A}^{\mathbb{N}_+}: \mathbf{1}_{t|_n}^{[\varnothing]}(q,\epsilon)=1\} \ \text{ and }\  \mathscr{A}^{\mathbb{N}_+}_{n}(q,\epsilon)^c=\mathscr{A}^{\mathbb{N}_+}\setminus \mathscr{A}^{\mathbb{N}_+}_n(q,\epsilon).$$

In~\cite{BJ} we proved the following result:

\begin{proposition}\label{prop}
Let $K$ be a compact subset of $J$. Then for any $\epsilon> 0$ there exist constants $C=C(K)>0$ and $\delta=\delta(\epsilon,K)>0$ such that for any $n\ge 1$,
\begin{equation}
\mathbb{E}\Big(\sup_{q\in K} \sum_{w\in \mathscr{A}^{n} } \mu_q\big([w]\cap \mathscr{A}^{\mathbb{N}_+}_{n}(q,\epsilon)^c\big)\Big)\le C\cdot n\cdot b^{-n\delta}.
\end{equation}
\end{proposition}

Now for $n\ge 1$ and $\epsilon>0$ we define the random Cantor-like sets in $\mathscr{A}^{\mathbb{N}_+}$
\begin{equation}\label{Xin}
\mathcal{C}_n(q,\epsilon)= \bigcap_{p\ge n} \mathscr{A}^{\mathbb{N}}_p(q,\epsilon) \text{ and } \mathcal{C}(q)=\lim_{\epsilon\to 0}\lim_{n\to\infty}\mathcal{C}_n{(q,\epsilon)}.
\end{equation}
Then we can deduce from Proposition~\ref{prop} that, with probability 1, for all $q\in K$, $\mu_q$ is carried by $\mathcal{C}(q)$, that is, $\mu_q(\mathcal{C}(q))=\|\mu_q\|=Y_q>0$.

It worth noting that by construction, for any $t\in \mathcal{C}(q)$, we have
\[
\lim_{n\to\infty} \frac{\log W_{t|_n}(\varnothing)}{-n}=\lim_{n\to\infty} \frac{\log W_{t|_n^-}(\varnothing)}{-n}=\lim_{n\to\infty} \frac{\log W_{t|_n^+}(\varnothing)}{-n}=\xi(q),
\]
\[
\lim_{n\to\infty} \frac{\log L_{t|_n}(\varnothing)}{-n}=\lim_{n\to\infty} \frac{\log L_{t|_n^-}(\varnothing)}{-n}=\lim_{n\to\infty} \frac{\log L_{t|_n^+}(\varnothing)}{-n}=\widetilde{\xi}(q),
\]
\[
\lim_{n\to\infty} \frac{\log O_U(t|_n)}{-n}=\lim_{n\to\infty} \frac{\log O_U(t|_n^-)}{-n}=\lim_{n\to\infty} \frac{\log O_U(t|_n^+)}{-n}=0, \ U\in\{W,L\}.
\]
Moreover, due to \eqref{osc1}, the above equalities imply that
\[
\lim_{r\to \infty} \frac{\log \mathrm{Osc}_F(B(\lambda(t),r))}{\log r} = \xi(q)/\widetilde{\xi}(q)=\tau'(q).
\]

\subsection{Proof of Theorem~\ref{lower}.}\label{proofthm2}

\begin{proof}

From now on, for $U\in\{W,L\}$, $w\in\mathscr{A}^*$ and $t\in \mathscr{A}^*\cup\mathscr{A}^{\mathbb{N}_+}$, we will use the convention $F_U^{[w]}\circ\lambda(t)=F_U^{[w]}(t)$.

For any $s,t\in\mathscr{A}^*\cup\mathscr{A}^{\mathbb{N}_+}$ and $\gamma>0$ we define the Riesz-like kernels:
\begin{equation}\label{kernel}
\mathcal{K}_{\gamma}(s,t)=
\left\{
\begin{array}{ll}
(|F_L(s)-F_L(t)|^2+|F_W(s)-F_W(t)|^2)^{-\gamma/2}\vee 1, & \textrm{ if } \gamma\ge 1;\\
\ \\
|F_W(s)-F_W(t)|^{-\gamma}\vee 1, & \textrm{ if }\gamma< 1.
\end{array}
\right.
\end{equation}

Recall the definitions of $\gamma^G(q)$ and $\gamma^R(q)$ in~\eqref{gammaG} and~\eqref{gammaR} (see also Remark~\ref{J12}).

For $q\in J$, $S\in\{G,R\}$ and $\delta>0$ we will use the notation
$$\mathcal{K}^{S}_{q,\delta}(s,t)=\mathcal{K}_{\gamma^S(q)-\delta}(s,t).$$

Recall the definition of $\mathcal{C}_n(q,\epsilon)$ in \eqref{Xin}. For $q\in J$, $\epsilon>0$ and $\delta>0$ define the $n$-th energy for $n\ge 1$ and $S\in\{G, R\}$:
\begin{equation}
\mathcal{I}^{S}_{n,\delta}(q,\epsilon)=\iint_{s,t\in \mathcal{C}_n(q,\epsilon), s\neq t} \mathcal{K}^{S}_{q,\delta}(s,t) \ \mathrm{d} \mu_q(s) \mathrm{d} \mu_q(t).
\end{equation}

Let $K$ be any compact subset of $J$. We assume for a while that we have proved that there exists $\delta_K>0$ such that for any $\delta\in(0,\delta_K)$, there exists $\epsilon_\delta>0$ such that for any $n\ge 1$, $\epsilon\in (0,\epsilon_\delta)$ and $S\in \{G, R\}$,
\begin{equation}\label{IGR}
\mathbb{E}\left(\sup_{q\in K} \mathcal{I}^{S}_{n,\delta}(q,\epsilon) \right)<\infty.
\end{equation}

The following lemma is a slight modification of Theorem 4.13 in~\cite{Falc} regarding the Hausdorff dimension estimate through the potential theoretic method.

\begin{lemma}\label{ptm}
Let $\mu$ be a Borel measure on $\mathbb{R}^m$ and let $E\subset\mathbb{R}^m$ be a Borel set such that $\mu(E)>0$. For any $\gamma>0$, if 
$$
\iint_{x,y\in E, x\neq y} |x-y|^{-\gamma}\vee 1 \ \mathrm{d}\mu(x)\mathrm{d}\mu(y)<\infty,
$$
then
$$\mu\left(\left\{ x\in E : h_\mu(x)= \liminf_{r\to 0^+} \frac{\log \mu(B(x,r))}{\log r}< \gamma\right\}\right)=0.$$
\end{lemma}

Then, it easily follows from Proposition~\ref{prop}, \eqref{IGR} and Lemma~\ref{ptm} that, with probability 1, for all $q\in K$:
\begin{itemize}
\item for $\mu_q^{G}$-almost every $x\in G_{q}:=\{(F_L(t),F_W(t)):t\in  \mathcal{C}(q) \}\subset G_F(\tau'(q))$,
$$
h_{\mu_q^G}(x)=\liminf_{r\to 0^+} \frac{1}{\log r}\log \mu^{G}_q(B(x,r)) \ge \gamma^{G}(q)-\delta;
$$
\item for $\mu_q^{R}$-almost every $y\in R_{q}:=\{F_W(t):t\in \mathcal{C}(q) \}\subset R_F(\tau'(q))$,
$$
h_{\mu_q^R}(x)=\liminf_{r\to 0^+} \frac{1}{\log r}\log \mu^{R}_q(B(y,r)) \ge \gamma^{R}(q)-\delta.
$$
\end{itemize}

We can consider a countable sequence of compact subintervals $K_{n}\subset J$ such that $\bigcup K_{n}=J$ and a corresponding sequence $\delta_n\in(0,\delta_{K_n})$. Then the above facts imply that with probability $1$, for any $q\in J$ and $S\in\{G,R\}$, for $\mu_q^S$-almost every $x\in S_q$, $h_{\mu_q^S}(x)\ge \gamma^S(q)$, hence $\dim_H(\mu_q^S)\ge \gamma^S(q)$ (we use the mass distribution principle, see \cite{Falc}).

\medskip

To complete the proof, we use the fact that, with probability $1$, for all $q\in J$, $\mu_q^{D}$ is carried by the set $E_F(\tau'(q))$. Then, applying Theorem~\ref{upper} to any set $E\subset\mathrm{Supp} (\mu_q^{D})\cap E_F(\tau'(q))$ yields
\begin{eqnarray*}
\dim_H (\mu_q^{G}) &\le& \Big( \frac{\dim_H (\mu_q^{D})}{\tau'(q)} \wedge \big(\dim_H (\mu_q^{D})+1-\tau'(q)\big)\Big) \vee \dim_H (\mu_q^{D}),\\
\dim_H (\mu_q^{R}) &\le&  \frac{\dim_H (\mu_q^{D})}{\tau'(q)} \wedge 1,
\end{eqnarray*}
and the conclusion comes from the fact that $\dim_H (\mu_q^{D})=\tau^*(\tau'(q))$ for all $q\in J$.

\medskip

Now we prove \eqref{IGR}.

\medskip

For any $\bar{q}\in K$ and $\epsilon>0$ we define the neighborhood of $\bar{q}$ in $K$:
\begin{equation}\label{Ulambda}
U_\epsilon(\bar{q})=\left\{q\in K: \max_{\alpha\in\{\xi,\widetilde{\xi},\tau,\tau',\gamma,\gamma^{G},\gamma^{R}\}} |\alpha(q)-\alpha(\bar{q})| < \epsilon \right\}.
\end{equation}
By continuity of these functions, the set $U_\epsilon(\bar{q})$ is open in $K$.

For any $u,v\in \mathscr{A}^*\cup\mathscr{A}^{\mathbb{N}_+}$ and $p\ge 2$, we define the indicator function
$$\mathbf{1}_p(u,v)=\mathbf{1}_{\{b^{-p+1}\le |\lambda(u)-\lambda(u)|<b^{-p+2}\}}.$$

For any $w\in\mathscr{A}^*$ we use the notation 
$$
[w]^n_{q,\epsilon}=[w]\cap\mathcal{C}_n(q,\epsilon).
$$ 
Notice that for $q\in K$, $\delta>0$ and $S\in\{G,R\}$ the Riez-like kernels $\mathcal{K}^{S}_{q,\delta}$ is a positive function and, moreover, by the continuity of $F_W$ and $F_L$ we have for any $s,t\in\mathscr{A}^{\mathbb{N}_+}$,
\[
\lim_{m\to \infty}\mathcal{K}^{S}_{q,\delta}(s|_m,t|_m)=\mathcal{K}^{S}_{q,\delta}(s,t).
\]
Then by applying Fatou's lemma we get
\begin{eqnarray*}
&&\mathcal{I}^{S}_{n,\delta}(q,\epsilon) \\
&=& \iint_{s,t\in \mathcal{C}_n(q,\epsilon),s\neq t} \lim_{m\to\infty} \mathcal{K}^{S}_{q,\delta}(s|_m,t|_m)\  \mathrm{d} \mu_q(s) \mathrm{d} \mu_q(t) \\
&=&\sum_{p\ge 2} \iint_{s,t\in \mathcal{C}_n(q,\epsilon); \mathbf{1}_p(s,t)=1} \lim_{m\to \infty} \mathcal{K}^{S}_{q,\delta}(s|_m,t|_m)\  \mathrm{d} \mu_q(s) \mathrm{d} \mu_q(t) \\
&\le& \sum_{p\ge 2} \liminf_{m\to \infty}  \iint_{s,t\in \mathcal{C}_n(q,\epsilon); \mathbf{1}_p(s,t)=1}  \mathcal{K}^{S}_{q,\delta}(s|_m,t|_m)\  \mathrm{d} \mu_q(s) \mathrm{d} \mu_q(t) \\
&= &\sum_{p\ge 2} \liminf_{m\to \infty} \sum_{u,v\in \mathscr{A}^m; \mathbf{1}_p(u,v)=1} \mathcal{K}^{S}_{q,\delta}(u,v)\cdot \mu_q([u]^n_{q,\epsilon}) \mu_q([v]^n_{q,\epsilon})\\
&\le& \sum_{p\ge 2} \liminf_{m\to \infty} \sum_{u,v\in \mathscr{A}^m; \mathbf{1}_p(w,u)=1} \mathcal{K}^{S}_{\bar{q},\delta+\epsilon}(u,v)\cdot \mu_q([u]^n_{q,\epsilon}) \mu_q([v]^n_{q,\epsilon}),
\end{eqnarray*}
where the last inequality comes from the fact that due to~\eqref{kernel} and~\eqref{Ulambda}, for any $\bar{q}\in K$, $\epsilon>0$ and $u,v\in\mathscr{A}^*\cup\mathscr{A}^{\mathbb{N}_+}$, we have $\sup_{q\in U_\epsilon(\bar{q}) }  \mathcal{K}^{S}_{ q,\delta}(u,v) \le \mathcal{K}^{S}_{\bar q,\delta+\epsilon}(u,v)$. Let 
$$
A_{p,m}=\sum_{u,v\in \mathscr{A}^m; \mathbf{1}_p(u,v)=1} \mathcal{K}^{S}_{\bar q,\delta+\epsilon}(u,v)\cdot \mu_q([u]^n_{q,\epsilon}) \mu_q([v]^n_{q,\epsilon}).
$$
Then,
\begin{eqnarray}
\sup_{q\in U_\epsilon(\bar{q}) }  \mathcal{I}^{S}_{n,\delta}(q,\epsilon) &\le& \sup_{q\in U_\epsilon(\bar{q}) } \sum_{p\ge 2} \liminf_{m\to \infty}\ A_{p,m}  \nonumber \\
&\le&  \sup_{q\in U_\epsilon(\bar{q}) } \sum_{p\ge 2} \Big( A_{p,m_p}+\sum_{m\ge m_p} |A_{p,m+1}-A_{p,m}|\Big) \nonumber \\ 
&\le&  \sum_{p\ge 2} \Big(  \sup_{q\in U_\epsilon(\bar{q}) }A_{p,m_p}+\sum_{m\ge m_p}  \sup_{q\in U_\epsilon(\bar{q}) }|A_{p,m+1}-A_{p,m}|\Big), \label{uniupper}
\end{eqnarray}
where for $p\ge 2$, we can choose $m_p\ge 3$ to be any integer. We have
\begin{equation}\label{IdeltaI}
\sup_{q\in U_\epsilon(\bar{q}) } A_{p,m} \le (\mathcal{I}^{S,\delta}_{n,\bar{q},\epsilon})_{p,m} \quad\text{and}\quad\sup_{q\in U_\epsilon(\bar{q}) }|A_{p,m+1}-A_{p,m}|\le (\Delta \mathcal{I}^{S,\delta}_{n,\bar{q},\epsilon})_{p,m},
\end{equation}
where
\begin{eqnarray*}
(\mathcal{I}^{S,\delta}_{n,\bar{q},\epsilon})_{p,m}&=& \sum_{u,v\in \mathscr{A}^m;\mathbf{1}_p(u,v)=1} \mathcal{K}^{S}_{\bar q,\delta+\epsilon}(u,v) \sup_{q\in U_\epsilon(\bar{q}) } \mu_q([u]^n_{q,\epsilon}) \mu_q([v]^n_{q,\epsilon}),\\
 (\Delta \mathcal{I}^{S,\delta}_{n,\bar{q},\epsilon})_{p,m} &=&\sum_{u,v\in \mathscr{A}^m; u',v'\in\mathscr{A}; \mathbf{1}_p(u,v)=1} \big|\mathcal{K}^{S}_{\bar q,\delta+\epsilon}(uu',vv')-\mathcal{K}^{S}_{\bar q,\delta+\epsilon}(u,v)\big|\cdot \\
 && \qquad\qquad\qquad\qquad\qquad\qquad\qquad\qquad \sup_{q\in U_\epsilon(\bar{q}) } \mu_q([uu']^n_{q,\epsilon}) \mu_q([vv']^n_{q,\epsilon}),
\end{eqnarray*}
and we have used the equality $\mu_q([u]^n_{q,\epsilon})=\sum_{u'\in\mathscr{A}} \mu_q([uu']^n_{q,\epsilon})$ to get the second inequality.

\begin{remark}\label{J12} $\ $
\begin{itemize}
\item[(a)] For technical reasons, we need to divide $J$ into three parts, in which $K$ will be chosen: 
$$
\begin{cases}
J_1=\{q\in J : \gamma^{G}(q)>1\},\\
J_2=\{q\in J : \gamma^{G}(q)\le 1, \ \tau'(q)< 1\},\\
J_3= \{q\in J : \gamma^{G}(q)\le 1, \ \tau'(q)\ge  1\}.
\end{cases}
$$
Then, due to \eqref{gammaG} and~\eqref{gammaR}, with $h=\tau'(q)$ we have 
\begin{equation*}
\gamma^G(q)=\left\{
\begin{array}{ll}
\tau^*(h)+1-h & \textrm{if } q\in J_1,\\
\displaystyle \tau^*(h)/h & \textrm{if } q\in J_2, \\
\tau^*(h)&\textrm{if } q\in J_3
\end{array}
\right. \text{ and } 
\gamma^R(q)=\left\{
\begin{array}{ll}
1& \textrm{if } q\in J_1,\\
 \tau^*(h)/h & \textrm{if } q\in J_2\cup J_3.
\end{array}
\right.
\end{equation*}
\item[(b)] Here we briefly explain why $\dim_H(\mu_q^G)=\gamma^{G}(q)=\tau^*(\tau'(q))$ when $q\in J_3$. Thus we will not consider this case in the rest of the proof. 

\noindent
From~\cite{BJ} we have $\dim_H(\mu_q^D)=\tau^*(\tau'(q))$ and $\mu^D_q$ is the orthogonal projection of $\mu^G_q$ onto the $x$-axis, so we automatically have $\dim_H(\mu_q^G) \ge \dim_H(\mu_q^D)=\tau^*(\tau'(q))$. Then to prove $\dim_H(\mu_q^G)=\gamma^{G}(q)$ we only need an upper bound estimate, but this estimate is actually a direct consequence of Theorem~\ref{upper}.
\end{itemize}
\end{remark}

For any compact subset $K$ of $J$ there exists $c_K\in(0,1)$ such that for any $c<c_K$, $\gamma^{G}(q)-c>1$ if $K\subset J_1$ and $\gamma^{R}(q)-c>0$ if $K\subset J_2\cup J_3$. Let $\delta_K=\epsilon_K=c_K/2$.

An essential tool in this paper is the following proposition, whose proof is given in Section~\ref{forth}. 

\begin{proposition}\label{mainprop}
Let $S\in\{G,R\}$. Suppose that  $K$ is a compact subset of $J_1$ or $J_2$ if $S=G$, or a compact subset of $J_1$ or $J_2\cup J_3$ if $S=R$. Then there exists $\epsilon_*\in(0,\epsilon_K)$ such that for any $0<\delta<\delta_K$, we can find constants $\kappa_1,\kappa_2, \eta_1,\eta_2>0$ and $C>0$ such that for any $\bar{q}\in K$, $0<\epsilon\le \epsilon_*$, $n\ge 1$, $p\ge 2$, and $m\ge 3\cdot (n\vee p)$
\begin{eqnarray*}
\mathbb{E}\left((\mathcal{I}^{S,\delta}_{n,\bar{q},\epsilon})_{p,m}\right) &\le& C\cdot b^{(n\vee p) -p+1} \cdot e^{\mathbf{1}_{\{p<n\}}\kappa_1\cdot n}\cdot e^{-\eta_1\delta \cdot (n\vee p)+\kappa_1\epsilon\cdot m}; \\
\mathbb{E}\left((\Delta \mathcal{I}^{S,\delta}_{n,\bar{q},\epsilon})_{p,m}\right) &\le& C\cdot b^{(n\vee p) -p+1} \cdot e^{\kappa_2\cdot (n\vee p)-\eta_2\cdot m}.
\end{eqnarray*}
\end{proposition}

Now we may choose $m_p=\frac{\kappa_2+\frac{1}{2}\delta\eta_1}{\eta_2}\cdot (n\vee p)$ (by modifying a little $\eta_2$ we can always assume that $\frac{\kappa_2+\frac{1}{2}\delta\eta_1}{\eta_2}>3$) and $\epsilon_\delta=\epsilon_* \wedge\frac{\frac{1}{2}\delta\eta_1\eta_2}{\kappa_1(\kappa_2+\frac{1}{2}\delta\eta_1)}$. Then by using Proposition~\ref{mainprop} and~\eqref{uniupper},~\eqref{IdeltaI}, for any $\delta<\delta_K$, $\bar{q}\in K$, $\epsilon<\epsilon_\delta$ and $S\in\{G,R\}$ we have
\begin{eqnarray*}
&&\mathbb{E}\left(\sup_{q\in U_\epsilon(\bar{q}) }  \mathcal{I}^{S}_{n,\delta}(q,\epsilon)\right)\\
&\le& \sum_{p\ge 2} \left(  \mathbb{E}\left((\mathcal{I}^{S,\delta}_{n,\bar{q},\epsilon})_{p,m_p}\right)+\sum_{m\ge m_p}  \mathbb{E}\left((\Delta \mathcal{I}^{S,\delta}_{n,\bar{q},\epsilon})_{p,m}\right)\right) \\
&\le& \sum_{p\ge 2} C\cdot b^{(n\vee p) -p+1} \cdot\left(  e^{\mathbf{1}_{\{p<n\}}\kappa_1\cdot n}\cdot e^{-\eta_1\delta \cdot (n\vee p)+\kappa_1\epsilon\cdot m_p} + \sum_{m\ge m_p} e^{\kappa_2\cdot (n\vee p)-\eta_2\cdot m} \right)\\
&\le & C\cdot \sum_{p\ge 2}  b^{(n\vee p) -p+1} \cdot\left(  e^{\mathbf{1}_{\{p<n\}}\kappa_1\cdot n}\cdot e^{-\eta_1\delta \cdot (n\vee p)+\kappa_1\cdot \epsilon_\delta\cdot m_p} + e^{\kappa_2\cdot (n\vee p)-\eta_2\cdot m_p}\cdot\frac{1}{1-e^{-\eta_2}} \right)\\
&\le& \frac{C}{1-e^{-\eta_2}}\cdot \sum_{p\ge 2} b^{(n\vee p)-p+1}\cdot e^{\mathbf{1}_{\{p<n\}}\kappa_1\cdot n} \cdot\\
&&\qquad \qquad \ \ \left(e^{-\eta_1\delta \cdot (n\vee p)+\kappa_1\cdot \frac{\frac{1}{2}\delta\eta_1\eta_2}{\kappa_1(\kappa_2+\frac{1}{2}\delta\eta_1)}\cdot \frac{\kappa_2+\frac{1}{2}\delta\eta_1}{\eta_2}\cdot (n\vee p)}+e^{\kappa_2\cdot (n\vee p)-\eta_2\cdot \frac{\kappa_2+\frac{1}{2}\delta\eta_1}{\eta_2}\cdot (n\vee p)} \right)\\
&=& \frac{2C}{1-e^{-\eta_2}}\cdot \sum_{p\ge 2} b^{(n\vee p)-p+1}\cdot e^{\mathbf{1}_{\{p<n\}}\kappa_1\cdot n} \cdot e^{-\frac{1}{2}\delta \eta_1(n\vee p)}\\
&=& \frac{2C}{1-e^{-\eta_2}}\cdot \left(\sum_{p= 1}^{n-1} b^{n-p+1}\cdot e^{\kappa_1\cdot n}\cdot e^{-\frac{1}{2}\delta \eta_1n} + \sum_{p\ge n} b \cdot e^{-\frac{1}{2}\delta \eta_1 p}\right) \\
&=& \frac{2C}{1-e^{-\eta_2}}\cdot \left(b^{n+1}\cdot e^{\kappa_1\cdot n}\cdot e^{-\frac{1}{2}\delta \eta_1n} \cdot \frac{1-b^{-n}}{1-b^{-1}}  + b \cdot e^{-\frac{1}{2}\delta \eta_1 n}\cdot \frac{1}{1-e^{-\frac{1}{2}\delta \eta_1}}\right) <\infty.
\end{eqnarray*}

Since for any $0<\epsilon<\epsilon_\delta$, the family $\{U_\epsilon(\bar{q})\}_{\bar{q}\in K}$ forms an open covering of $K$, there exist $\bar{q}_1,\cdots,\bar{q}_N$ such that $\{U_\epsilon(\bar{q}_i)\}_{1\le i\le N}$ also covers $K$. This gives us the conclusion. \end{proof}

\subsection{Proof of Theorem~\ref{levelsets}.}

\begin{proof}

The proof of Theorem~\ref{levelsets} exploits the main idea developed in~\cite{Mat} to study the dimension of projections and sections of sets.  Some complications come from the fact that we want results holding for uncountably many sets and measures simultaneously. 

Through the proof we use the same notation as in Section~\ref{proofthm2}. Moreover, for $n\ge 1$, $q\in J$ and $\epsilon>0$ we define
$$
G_n(q,\epsilon)=\big\{(F_L(t),F_W(t)): t\in \mathcal{C}_n(q,\epsilon) \big\}
$$
and for $\theta\in (-\pi/2,\pi/2)$ we define $R_{n,\theta}(q,\epsilon)=\mathrm{Proj}_\theta(G_n(q,\epsilon)) \subset l_{\theta}$.

For any $y\in l_{\theta}$ we define the lower derivative of the measure $\mu_{q,\theta}^{R}\big|_{R_{n,\theta}(q,\epsilon)}$ with respect to the one-dimensional Lebesgue measure on $l_\theta$ at $y$:
$$
\underline{D}(\mu_{q,\theta}^{R}\big|_{R_{n,\theta}(q,\epsilon)},y)=\liminf_{r\to 0^+} \frac{1}{r}\cdot \mu_{q,\theta}^{R}\Big(B(y,r)\cap R_{n,\theta}(q,\epsilon)\Big).
$$

We fix a compact subset $K\subset J_1$ (recall Remark~\ref{J12}). For any $\bar{q}\in K$, we can choose $\delta\in(0,\delta_K)$ and $\epsilon_*$ such that the conclusions of Proposition~\ref{mainprop} hold. Notice that for such $\delta$ and $\epsilon\in(0,\epsilon_*)$ we always have $\gamma^G(\bar q)-\delta-\epsilon>1$. 

For $s,t\in\mathscr{A}^*\cup\mathscr{A}^{\mathbb{N}_+}$ recall
\[
\mathcal{K}_{1}(s,t)^{-1}=\left(|F_L(s)-F_L(t)|^2+|F_W(s)-F_W(t)|^2\right)^{\frac{1}{2}} \wedge 1,
\]
and for $\theta\in (-\pi/2,\pi/2)$ and $\gamma>0$ we define
$$
d_{\theta,\gamma}^G(s,t)=\mathcal{K}_{\gamma}(s,t)^{-1}\cdot |\sin(\theta+\theta_{s,t})|,
$$
where $\theta_{{s,t}}$ stands for the angle between $(F_L(t),F_W(t))-(F_L(s),F_W(s))$ and $x$-axis (clockwise). Notice that for any $r>0$ and $\gamma> 1$ we always have
$$\mathbf{1}_{\{d_{\theta,1}^G(s,t)<r\}}\le \liminf_{m\to\infty} \mathbf{1}_{\{d_{\theta,1}^G(s|_m,\,t|_m)<2r\}} \le \liminf_{m\to\infty} \mathbf{1}_{\{d_{\theta,\gamma}^G(s|_m,\,t|_m)<2r\}}.$$

Now, recall~\eqref{Ulambda}. For simplicity we will also use the notation
\begin{equation}\label{Mnq}
M^n_{\bar q,\epsilon}(u,v)=\sup_{q\in U_\epsilon(\bar{q}) } \mu_q([u]^n_{q,\epsilon}) \mu_q([v]^n_{q,\epsilon}).
\end{equation}
Then an integration similar to that used in the proof of Theorem 9.7 of~\cite{Mat}, as well as arguments similar to that used in Section~\ref{proofthm2} yield for any $\gamma>1$ (by using Fatou's lemma repetitively) 
\begin{eqnarray*}
&&\int_{0}^{\pi}\sup_{q\in U_\epsilon(\bar{q})} \int_{y\in R_{n,\theta}(q,\epsilon)}\underline{D}(\mu_{q,\theta}^{R}|_{R_{n,\theta}(q,\epsilon)},y)\mathrm{d} \mu_{q,\theta}^{R}(y) \mathrm{d}\theta \\
&=&\int_{0}^{\pi}\sup_{q\in U_\epsilon(\bar{q})} \int_{y\in R_{n,\theta}(q,\epsilon)}\liminf_{r\to 0^+} \frac{1}{r}\mu_{q,\theta}^{R}\Big(B(y,r)\cap R_{n,\theta}(q,\epsilon)\Big) \mathrm{d} \mu_{q,\theta}^{R}(y) \mathrm{d}\theta \\
&\le &\int_{0}^{\pi}\sup_{q\in U_\epsilon(\bar{q})} \liminf_{r\to 0^+} \frac{1}{r} \int_{y\in R_{n,\theta}(q,\epsilon)} \int_{x\in G_n(q,\epsilon)}\mathbf{1}_{\{ |x-l_{y,\theta}|\le r\}} \mathrm{d}\mu^{G}_q(x)\mathrm{d} \mu_{q,\theta}^{R}(y) \mathrm{d}\theta\\
&=&\int_{0}^{\pi}\sup_{q\in U_\epsilon(\bar{q})} \liminf_{r\to 0^+} \frac{1}{r} \int_{s\in \mathcal{C}_{n}(q,\epsilon)} \int_{t\in \mathcal{C}_n(q,\epsilon)}\mathbf{1}_{\{d^{G}_{\theta,1}(s,t)\le r\}} \mathrm{d}\mu_q(t)\mathrm{d} \mu_{q}(s) \mathrm{d}\theta\\
&\le&\int_{0}^{\pi}\sup_{q\in U_\epsilon(\bar{q})} \liminf_{r\to 0^+} \frac{1}{r}\sum_{p\ge 2} \iint_{s,t\in \mathcal{C}_{n}(q,\epsilon),\atop \mathbf{1}_p(s,t)=1}  \liminf_{m\to\infty} \mathbf{1}_{\{d_{\theta,1}^G(s|_m,t|_m)<2r\}}  \mathrm{d}\mu_q(t)\mathrm{d} \mu_{q}(s) \mathrm{d}\theta\\
&\le & \int_{0}^{\pi}\sup_{q\in U_\epsilon(\bar{q})} \liminf_{r\to 0^+} \frac{1}{r} \sum_{p\ge 2}  \iint_{s,t\in \mathcal{C}_{n}(q,\epsilon),\atop \mathbf{1}_p(s,t)=1 }\liminf_{m\to \infty} \mathbf{1}_{\{d_{\theta,\gamma}^G(s|_m,t|_m)<2r\}}  \mathrm{d}\mu_q(t)\mathrm{d} \mu_{q}(s) \mathrm{d}\theta\\
&\le & \int_{0}^{\pi}\sup_{q\in U_\epsilon(\bar{q})} \liminf_{r\to 0^+} \frac{1}{r} \sum_{p\ge 2}  \liminf_{m\to \infty}  \iint_{s,t\in \mathcal{C}_{n}(q,\epsilon),\atop \mathbf{1}_p(s,t)=1 }\mathbf{1}_{\{d_{\theta,\gamma}^G(s|_m,t|_m)<2r\}}  \mathrm{d}\mu_q(t)\mathrm{d} \mu_{q}(s) \mathrm{d}\theta\\
&\le& \liminf_{r\to 0^+} \frac{1}{r} \sum_{p\ge 2} \int_{0}^{\pi}\sup_{q\in U_\epsilon(\bar{q})} \liminf_{m\to \infty}  \sum_{u,v\in\mathscr{A}^m; \atop \mathbf{1}_p(s,t)=1} \mathbf{1}_{\{d_{\theta,\gamma}^G(u,v)<2r\}} \mu_q([u]^n_{q,\epsilon})\mu_{q}([v]^n_{q,\epsilon}) \mathrm{d}\theta \\
&\le & \liminf_{r\to 0^+} \frac{1}{r}\sum_{p\ge 2} \Bigg( \sum_{u,v\in\mathscr{A}^{m_p}; \atop \mathbf{1}_p(s,t)=1} \int_{0}^{\pi} \mathbf{1}_{\{d_{\theta,\gamma}^G(u,v)<2r\}}  \mathrm{d}\theta \cdot M^n_{\bar q,\epsilon}(u,v)  +  \sum_{m\ge m_p}\\
& & \sum_{u, v\in \mathscr{A}^m,u',v'\in\mathscr{A}; \atop \mathbf{1}_p(u,v)=1} \int_{0}^{\pi} \big|\mathbf{1}_{\{d_{\theta,\gamma}^G(uu',vv')<2r\}}-\mathbf{1}_{\{d_{\theta,\gamma}^G(u,v)<2r\}}\big|\mathrm{d}\theta  \cdot M^n_{\bar q,\epsilon}(uu',vv') \Bigg)\\
&\le & \sum_{p\ge 2} \Bigg( \sum_{u,v\in\mathscr{A}^{m_p}; \atop \mathbf{1}_p(u,v)=1}  \limsup_{r\to 0^+} \frac{1}{r} \int_{0}^{\pi} \mathbf{1}_{\{d_{\theta,\gamma}^G(u,v)<2r\}}  \mathrm{d}\theta \cdot M^n_{\bar q,\epsilon}(u,v) + \sum_{m\ge m_p}  \\
& & \sum_{u, v\in \mathscr{A}^m,u',v'\in\mathscr{A}; \atop \mathbf{1}_p(u,v)=1}  \limsup_{r\to 0^+} \frac{1}{r} \int_{0}^{\pi}\big|\mathbf{1}_{\{d_{\theta,\gamma}^G(uu',vv')<2r\}}-\mathbf{1}_{\{d_{\theta,\gamma}^G(u,v)<2r\}} \big|\mathrm{d}\theta \cdot M^n_{\bar q,\epsilon}(uu',vv')\Bigg),
\end{eqnarray*}
where $m_p$ is taken as in the proof of Theorem~\ref{lower}.

Notice that there exists a universal constant $C'>0$ such that for all $r>0$,
$$
\int_{0}^{\pi}  \mathbf{1}_{\{d_{\theta,\gamma}^G(w,u)\le2r\}}  \mathrm{d}\theta=\int_{0}^{\pi}  \mathbf{1}_{\{|\sin(\theta+\theta_{u,v})|\le 2r\cdot \mathcal{K}_\gamma(u,v)\}}  \mathrm{d}\theta \le C'\cdot r\cdot \mathcal{K}_\gamma(u,v)
$$
and
$$
\int_{0}^{\pi}\big|\mathbf{1}_{\{d_{\theta,\gamma}^G(uu',vv')\le2r\}}-\mathbf{1}_{\{d_{\theta,\gamma}^G(u,v)\le2r\}}\big|\mathrm{d}\theta\le C'\cdot r\cdot |\mathcal{K}_\gamma(uu',vv')-\mathcal{K}_\gamma(u,v)|.
$$
Thus by taking $\gamma=\gamma^G(\bar q)-\delta-\epsilon>1$, we deduce from Proposition~\ref{mainprop} that
\begin{eqnarray*}
&&\mathbb{E}\left (\int_{0}^{\pi}\sup_{q\in U_\epsilon(\bar{q})} \int_{y\in R_{n,\theta}(q,\epsilon)}\underline{D}(\mu_{q,\theta}^{R}\big|_{R_{n,\theta}(q,\epsilon)},y)\mathrm{d} \mu_{q,\theta}^{R}(y) \mathrm{d}\theta\right ) \\
&\le& C'\cdot \mathbb{E}\left (\sum_{p\ge 2} (\mathcal{I}^{S,\delta}_{n,\bar{q},\epsilon})_{p,m_p}+ \sum_{m\ge m_p} (\Delta \mathcal{I}^{S,\delta}_{n,\bar{q},\epsilon})_{p,m}\right )<\infty.
\end{eqnarray*}
Then by using the same argument as in Section~\ref{proofthm2} we can conclude that with probability 1, for Lebesgue almost every $\theta\in(-\pi/2,\pi/2)$, for all $q\in J_1$, for $\mu_{q,\theta}^{R}$-almost every $y\in l_{\theta}$, the lower derivative $\underline{D}(\mu_{q,\theta}^{R},y)$ is finite, which is equivalent to saying that $\mu_{q,\theta}^{R}$ is absolutely continuous with respect to the one-dimensional Lebesgue measure on $l_{\theta}$. This ensures that for $\mu_{q,\theta}^{R}$-almost every $y\in l_{\theta}$, the following limit:
\begin{equation*}
\lim_{r\to 0^+}\frac{1}{r} \int_{|x-l_{y,\theta}^\perp |\le r} \psi(x) \ \mathrm{d} \mu_q^{G}(x) 
\end{equation*}
exists for any continuous function $\psi:\mathbb{R}^2\mapsto \mathbb{R}_+$ and thus defines a measure $\mu^{y}_{q,\theta}$ carried by $l_{y,\theta}^\perp \cap \{(F_L(t),F_W(t)):t\in  \mathcal{C}(q) \}$.

Now for the lower bound of the Hausdorff dimension of $\mu^{y}_{q,\theta}$, we notice that like in Theorem 10.7 of~\cite{Mat}, we have the following equality for the $\gamma^G(q)-\delta-1$ energy
\begin{eqnarray*}
&& \mathcal{I}_{q,\theta}\\
&:=&\int_{y\in R_{n,\theta}(q,\epsilon)}\iint_{x_1,x_2\in l_{y,\theta}^\perp\cap G_n(q,\epsilon)} 1\vee |x_1-x_2|^{-\gamma^G(q)-\delta-1} \mathrm{d} \mu_{q,\theta}^{y}(x_1)\mathrm{d} \mu_{q,\theta}^{y}(x_2)\mathrm{d} \mathrm{Leb}_\theta(y)\\
&=&\liminf_{r\to 0^+} \frac{1}{r} \iint_{s,t\in \mathcal{C}_{n}(q,\epsilon)}\mathbf{1}_{\{d^{G}_{\theta,1}(s,t)\le r\}} d^{G}(s,t)^{-\gamma^G(q)-\delta-1}\mathrm{d}\mu_q(t)\mathrm{d} \mu_{q}(s),
\end{eqnarray*}
where $\mathrm{Leb}_\theta$ stands for the one-dimensional Lebesgue measure on $l_{\theta}$. By using the same method as above to establish the finiteness of $\underline{D}(\mu_{q,\theta}^{R},y)$, we can show that
\begin{multline*}
\int_{0}^{\pi}\sup_{q\in U_\epsilon(\bar{q})}  \mathcal{I}_{q,\theta} \ \mathrm{d}\theta\\
\le \sum_{p\ge 2} \Big( \sum_{u,v\in\mathscr{A}^{m_p}; \atop \mathbf{1}_p(u,v)=1}  \limsup_{r\to 0^+} \frac{1}{r} \int_{0}^{\pi}  \mathbf{1}_{\{d_{\theta,1}^G(u,v)\le2r\}}\mathcal{K}_{\widetilde{\gamma}}(u,v)\mathrm{d}\theta \cdot M^n_{\bar q,\epsilon}(u,v)+\sum_{m\ge m_p}  \sum_{u,v\in \mathscr{A}^m;u',v'\in\mathscr{A}; \atop \mathbf{1}_p(u,v)=1} \\
\limsup_{r\to 0^+} \frac{1}{r} \int_{0}^{\pi}\big|\mathbf{1}_{\{d_{\theta,1}^G(uu',vv')\le2r\}}\mathcal{K}_{\widetilde{\gamma}}(uu',vv')-\mathbf{1}_{\{d_{\theta,1}^G(w,u)\le2r\}}\mathcal{K}_{\widetilde{\gamma}}(u,v)\big|\mathrm{d}\theta \cdot M^n_{\bar q,\epsilon}(u,v) \Big),
\end{multline*}
where $\widetilde{\gamma}=\gamma-1=\gamma^G(\bar{q})-\delta-\epsilon-1>0$. 

\medskip

Notice that for the same universal constant $C'$, we have
$$
\int_{0}^{\pi}  \mathbf{1}_{\{d_{\theta,1}^G(u,v)\le2r\}}\mathcal{K}_{\widetilde{\gamma}}(u,v) \mathrm{d}\theta \le C'\cdot r\cdot \mathcal{K}_{\widetilde{\gamma}+1}(u,v)=C'\cdot r\cdot \mathcal{K}_{\gamma}(u,v)
$$
and
\begin{eqnarray*}
&&\int_{0}^{\pi}\big|\mathbf{1}_{\{d_{\theta,1}^G(uu',vv')\le2r\}}\mathcal{K}_{\widetilde{\gamma}}(uu',vv')-\mathbf{1}_{\{d_{\theta,1}^G(u,v)\le2r\}}\mathcal{K}_{\widetilde{\gamma}}(u,v)\big|\mathrm{d}\theta \\
&\le & \Big(\mathcal{K}_{\widetilde{\gamma}}(uu',vv') \vee \mathcal{K}_{\widetilde{\gamma}}(u,v) \Big)\cdot \int_{0}^{\pi}\big|\mathbf{1}_{\{d_{\theta,1}^G(uu',vv')\le2r\}}-\mathbf{1}_{\{d_{\theta,1}^G(u,v)\le2r\}}\big|\mathrm{d}\theta\\
&& \qquad +\ | \mathcal{K}_{\widetilde{\gamma}}(uu',vv')-\mathcal{K}_{\widetilde{\gamma}}(u,v) |\int_{0}^{\pi} \mathbf{1}_{\{d_{\theta,1}^G(uu',vv')\le2r\}}\vee \mathbf{1}_{\{d_{\theta,1}^G(u,v)\le2r\}} \mathrm{d}\theta\\
&\le& C'\cdot r\cdot \Delta \mathcal{K}_{\widetilde{\gamma}}(uu',vv'),
\end{eqnarray*}
where
\begin{multline*}
\Delta \mathcal{K}_{\widetilde{\gamma}}(uu',vv')=
\Big(\mathcal{K}_{\widetilde{\gamma}}(uu',vv')\vee \mathcal{K}_{\widetilde{\gamma}}(u,v) \Big)\cdot |\mathcal{K}_1(uu',vv')-\mathcal{K}_{1}(u,v)| \\
 +\Big(\mathcal{K}_{1}(uu',vv')\vee \mathcal{K}_{1}(u,v) \Big)\cdot |\mathcal{K}_{\widetilde{\gamma}}(uu',vv')-\mathcal{K}_{\widetilde{\gamma}}(u,v)|.
\end{multline*}
Then
\begin{equation}\label{boundenergy}
\mathbb{E}\left (\int_{0}^{\pi}\sup_{q\in U_\epsilon(\bar{q})}  \mathcal{I}_{q,\theta} \ \mathrm{d}\theta \right ) 
\le \mathbb{E}\left (\sum_{p\ge 2} (\mathcal{I}^{S,\delta}_{n,\bar{q},\epsilon})_{p,m_p}+ \sum_{m\ge m_p} \widetilde{(\Delta \mathcal{I}^{S,\delta}_{n,\bar{q},\epsilon})}_{p,m}\right )<\infty,
\end{equation}
where
\[
\widetilde{(\Delta \mathcal{I}^{S,\delta}_{n,\bar{q},\epsilon})}_{p,m}=\sum_{u,v\in \mathscr{A}^m;u',v'\in\mathscr{A}; \atop \mathbf{1}_p(u,v)=1} \Delta \mathcal{K}_{\widetilde{\gamma}}(uu',vv') \sup_{q\in U_\epsilon(\bar{q}) } \mu_q([uu']^n_{q,\epsilon}) \mu_q([vv']^n_{q,\epsilon}).
\]
The justification of the finiteness in \eqref{boundenergy} is postponed to Remark~\ref{diffkernel2}.  

Now, by using the same arguments as in Section~\ref{proofthm2} again, we deduce that, with probability 1, for Lebesgue-almost every $\theta\in (-\pi/2,\pi/2)$, for all $q\in J_1$, for $\mu_{q,\theta}^{R}$-almost every $y\in l_{\theta}$, we have
$$\dim_H (\mu^{y}_{q,\theta})\ge \dim_H (\mu_q^G)-1=\tau^*(\tau'(q))-\tau'(q).$$
Then we get the conclusion by applying Theorem~\ref{upper}(b) to the measure $\mu_{q,\theta}^{R}$ since we know that the for $\mu_{q,\theta}^{R}$-almost every $y\in l_{\theta}$, the lower local dimension $h_{\mu_{q,\theta}^{R}}(y)$ is equal to $1$. \end{proof}

\section{Proof of Proposition~\ref{mainprop}.} \label{forth}

\subsection{Main Proof.}

\begin{proof}

Due to~\eqref{Ulambda} we always have
$$\bigcup_{q\in U_\epsilon(\bar{q}) } \mathcal{C}_n(q,\epsilon) \subset \mathcal{C}_n(\bar{q},2\epsilon).$$

Let $\rho_K=1\vee \sup_{q\in K}\{|q|+\xi(q)+\widetilde{\xi}(q)+|\tau(q)|+\gamma(q)\}<\infty$. 

Due to~\eqref{Mnq},~\eqref{selfsimimeasure},~\eqref{indicator},~\eqref{Ulambda} and~\eqref{Y_K} we have
\begin{eqnarray*}
M^n_{\bar q,\epsilon}(u,v)&=&\sup_{q\in U_\epsilon(\bar{q}) } \mu_q([u]^n_{q,\epsilon}) \mu_q([v]^n_{q,\epsilon})\\
&=&\sup_{q\in U_\epsilon(\bar{q}) } \mathbf{1}_{\{[u]^n_{q,\epsilon}\neq\emptyset\}}\mathbf{1}_{\{[v]^n_{q,\epsilon}\neq\emptyset\}} |W_{q,u}(\varnothing)||W_{q,v}(\varnothing)| Y_q(u)Y_q(v)\\
&\le&\sup_{q\in U_\epsilon(\bar{q}) } \mathbf{1}_{\{[u]^n_{q,\epsilon},[v]^n_{q,\epsilon}\neq\emptyset\}}  e^{-(|u|+|v|)(\gamma(q)-2\rho_K\epsilon)}Y_q(u)Y_q(v)\\
&\le&  \mathbf{1}_{\{[u]^n_{\bar{q},2\epsilon},[v]^n_{\bar{q},2\epsilon}\neq\emptyset\}}  e^{-(|u|+|v|)(\gamma(\bar{q})-4\rho_K\epsilon)} Y_K(u)Y_K(v).
\end{eqnarray*}
This gives us 
\begin{multline}
(\mathcal{I}^{S,\delta}_{n,\bar{q},\epsilon})_{p,m} \le e^{-2m(\gamma(\bar{q})-4\rho_K\epsilon)} \cdot \\
 \sum_{u,v\in \mathscr{A}^m, \mathbf{1}_p(u,v)=1} \mathcal{K}^{S}_{\bar q,\delta+\epsilon}(u,v)\cdot \mathbf{1}_{\{[u]^n_{\bar{q},2\epsilon},[v]^n_{\bar{q},2\epsilon}\neq\emptyset\}} Y_K(u)Y_K(v), \label{S}
\end{multline}

\begin{multline}
(\Delta I^{S,\delta}_{n,\bar{q},\epsilon})_{p,m} \le e^{-2(m+1)(\gamma(\bar{q})-4\rho_K\epsilon)}\cdot \sum_{u,v\in \mathscr{A}^m; u',v'\in\mathscr{A}; \mathbf{1}_p(u,v)=1}\\
\big|\mathcal{K}^{S}_{\bar q,\delta+\epsilon}(uu',vv')-\mathcal{K}^{S}_{\bar q,\delta+\epsilon}(u,v)\big|\cdot\mathbf{1}_{\{[uu']^n_{\bar{q},2\epsilon},[vv']^n_{\bar{q},2\epsilon}\neq\emptyset\}} Y_K(uu')Y_K(vv').\label{T}
\end{multline}

Now we deal with the individual terms of the above sums.

Fix $p$ and $n$ in $\mathbb{N}_+$, let $r=p\vee n$, and fix $m\ge 3r$. 

Fix a pair $u,v\in\mathscr{A}^m$ with $\mathbf{1}_p(u,v)=1$, so $|\lambda(u)-\lambda(v)|\in[b^{-p+1},b^{-p+2})$.

Without loss of generality we suppose that $\lambda(u)<\lambda(v)$.

Since $\mathbf{1}_p(u,v)=1$, we have $\lambda(u)<\lambda(v|_{p}^-) \le \lambda(v|_r^-)$.

Let
\begin{equation}\label{VdeltaV}
\left\{
\begin{array}{l}
V:=\mathcal{K}^{S}_{\bar q,\delta+\epsilon}(u,v)\mathbf{1}_{\{[u]^n_{\bar{q},2\epsilon},[v]^n_{\bar{q},2\epsilon}\neq\emptyset\}} Y_K(u)Y_K(v); \\
\ \\
\Delta V:=\big|\mathcal{K}^{S}_{\bar q,\delta+\epsilon}(uu',vv')-\mathcal{K}^{S}_{\bar q,\delta+\epsilon}(u,v)\big|\mathbf{1}_{\{[uu']^n_{\bar{q},2\epsilon},[vv']^n_{\bar{q},2\epsilon}\neq\emptyset\}} Y_K(uu')Y_K(vv').
\end{array}
\right.
\end{equation}

\smallskip

Let us state two elementary claims.

\medskip

\textbf{Claim 1.} Recall that $[u]^n_{\bar{q},2\epsilon}=[u]\cap \mathcal{C}_n(\bar{q},2\epsilon)$. Due to~\eqref{indicator} and~\eqref{Xin}, if $[u]^n_{\bar{q},2\epsilon}\neq\emptyset$, then for $l=r,\cdots,m$ we have
$\mathbf{1}_{\mathscr{W}_{u|_l}(\bar{q},2\epsilon)}\cdot \mathbf{1}_{\mathscr{L}_{u|_l}(\bar{q},2\epsilon)}\cdot \mathbf{1}_{\mathscr{O}_{u|_l}(\epsilon)}=1$. Define
\begin{eqnarray}
\mathbf{1}^{(1)}_{u,v}(\bar q,\epsilon) &=& \mathbf{1}_{\left\{O_{F_W}(I_{u|_{r}})\vee O_{F_W}(I_{v|_{r}}) \le e^{-r(\xi(\bar{q})-6\epsilon)} ,\ O_{F_L}(I_{v|_{r}^-}) \ge e^{-r(\widetilde{\xi}(\bar{q})+6\epsilon)}\right\}}, \label{chi1}\\
\mathbf{1}^{(2)}_{u}(\bar q,\epsilon) &=& \mathbf{1}_{\mathscr{W}_{u|_r}(\bar{q},2\epsilon)}\mathbf{1}_{\mathscr{L}_{u|_r}(\bar{q},2\epsilon)}  \cdot \mathbf{1}_{\mathscr{W}_{u}(\bar{q},2\epsilon)} \mathbf{1}_{\mathscr{L}_{u}(\bar{q},2\epsilon)}\mathbf{1}_{\mathscr{O}_{u}(\epsilon)},\label{chi2}
\end{eqnarray}
and $\mathbf{1}^{(2)}_{u,v}(\bar q,\epsilon)=\mathbf{1}^{(2)}_{u}(\bar q,\epsilon)\cdot\mathbf{1}^{(2)}_{v}(\bar q,\epsilon)$. Also, $[uu']^n_{\bar{q},2\epsilon}\neq \emptyset$ implies $[u]^n_{\bar{q},2\epsilon}\neq\emptyset$. Then, due to~\eqref{osc1} and~\eqref{indicLWO}, we have
\begin{equation*}
\mathbf{1}_{\{[uu']^n_{\bar{q},2\epsilon},[vv']^n_{\bar{q},2\epsilon}\neq\emptyset\}}\le \mathbf{1}_{\{[u]^n_{\bar{q},2\epsilon},[v]^n_{\bar{q},2\epsilon}\neq\emptyset\}} \le \mathbf{1}^{(1)}_{u,v}(\bar q,\epsilon) \cdot \mathbf{1}^{(2)}_{u,v}(\bar q,\epsilon).
\end{equation*}

\medskip

\textbf{Claim 2.} For $w\in\mathscr{A}^*$ we define $Z_W(w)=F_W^{[w]}(1)$. Then from~(\ref{WL}) we have
\begin{equation}\label{ZW}
F_W(\lambda(w)+b^{-|w|})-F_W(\lambda(w))=W_{w}(\varnothing)\cdot Z_W(w), \ \ \forall \  w\in \mathscr{A}^*.
\end{equation}
Due to~\eqref{ZW} we have
\begin{eqnarray*}
&&F_W(v)-F_W(u) \\
&=& F_W(v)-F_W(v|_r)+F_W(v|_r)-F_W(v|_r^-)+F_W(v|_r^-)-F_W(u)\\
&=& W_{v|_r^-}(\varnothing)\cdot Z_W(v|_r^-)+ F_W(v|_r)-F_W(v|_r^-)+F_W(v|_r^-)-F_W(u).
\end{eqnarray*}

By construction we have $Z_W(v|_r^-)$ is measurable with respect to
\[
\mathcal{A}(v|_r^-):=\sigma\big((W,L)(v|_r^-\cdot w): w\in\mathscr{A}^*\big)
\]
and is independent of
$$\mathcal{A}^c(v|_r^-) :=\sigma\big((W,L)(w): w\in\mathscr{A}^*, |w|<r \text{ or } w|_r\neq w|_r^-\big).$$
Also due to the statistical self-similarity~\eqref{WL} we have
\[
\sigma
\left(\left\{\begin{array}{l}
W_{v|_r^-},F_W(v|_r)-F_W(v|_r^-)+F_W(v|_r^-)-F_W(u),\\
\mathbf{1}^{(2)}_{u,v}(\bar q,\epsilon),Y_K(u), Y_K(uu'),Y_K(v),Y_K(vv')
\end{array}
\right.\right) \subset \mathcal{A}^c(u|_r^-).
\]

\medskip

Now, due to \textbf{Claim 1} and~\eqref{VdeltaV} we have
\begin{eqnarray*}
V &\le& \overline{\mathcal{K}^{S}_{\bar q,\delta+\epsilon}}(u,v) \cdot \mathbf{1}^{(2)}_{u,v}(\bar q,\epsilon) \cdot Y_K(u)Y_K(v), \ \text{and}\\
\Delta V &\le& \overline{\Delta \mathcal{K}^{S}_{\bar q,\delta+\epsilon}}(uu',vv') \cdot \mathbf{1}^{(2)}_{u,v}(\bar q,\epsilon) \cdot Y_K(uu')Y_K(vv'),
\end{eqnarray*}
where
\begin{equation}\label{kdeltak}
\left\{
\begin{array}{l}
\overline{\mathcal{K}^{S}_{\bar q,\delta+\epsilon}}(u,v)=\mathcal{K}^{S}_{\bar q,\delta+\epsilon}(u,v)\cdot \mathbf{1}^{(1)}_{u,v}(\bar q,\epsilon); \\
\ \\
\overline{\Delta \mathcal{K}^{S}_{\bar q,\delta+\epsilon}}(uu',vv')=\big|\mathcal{K}^{S}_{\bar q,\delta+\epsilon}(uu',vv')-\mathcal{K}^{S}_{\bar q,\delta+\epsilon}(u,v)\big|\cdot \mathbf{1}^{(1)}_{u,v}(\bar q,\epsilon).
\end{array}
\right.
\end{equation}

Then due to \textbf{Claim 2} we have
\begin{equation}
\mathbb{E}\Big(V\Big| \mathcal{A}^c(v|_r^-)\Big) \le \mathbb{E}\Big( \overline{\mathcal{K}^{S}_{\bar q,\delta+\epsilon}}(u,v) \Big|\mathcal{A}^c(v|_r^-)\Big) \cdot \mathbf{1}^{(2)}_{u,v}(\bar q,\epsilon) \cdot Y_K(u) Y_K(v), \label{condiineq1}
\end{equation}
\begin{equation}
\mathbb{E}\Big(\Delta V \Big|\mathcal{A}^c(v|_r^-)\Big) \le \mathbb{E}\Big( \overline{\Delta \mathcal{K}^{S}_{\bar q,\delta+\epsilon}}(uu',vv') \Big|\mathcal{A}^c(v|_r^-)\Big) \cdot \mathbf{1}^{(2)}_{u,v}(\bar q,\epsilon) \cdot Y_K(uu')Y_K(vv').\label{condiineq2}
\end{equation}

Recall in Remark~\ref{J12} we distinguished the cases $K\subset J_i$, $i=1,2,3$ according to whether or not the corresponding power on the kernel is greater than $1$. Then, due to~\eqref{kernel}, once we have taken $\delta<\delta_K$ and $\epsilon<\epsilon_K$, only two situations are left:
\begin{equation*}
\mathcal{K}^{S}_{\bar q,\delta+\epsilon}(u,v)=\left\{
\begin{array}{ll}
(|F_L(u)-F_L(v)|^2+|F_W(u)-F_W(v)|^2)^{-\gamma/2} \vee 1, & \textrm{ if } \gamma> 1;\\
\ \\
|F_W(u)-F_W(v)|^{-\gamma} \vee 1, & \textrm{ if }\gamma< 1,
\end{array}
\right.
\end{equation*}
where $\gamma=\gamma^{S}(\bar{q})-\delta-\epsilon$.

We have the following lemma, whose proof is given in Section~\ref{proofpropker}.

\begin{lemma}\label{kernelcondi}
There exists a constant $C_\gamma$ such that
\begin{multline*}
\mathbf{1}^{(2)}_{w,u}(\bar q,\epsilon)\cdot  \mathbb{E}\Big( \overline{\mathcal{K}^{S}_{\bar q,\delta+\epsilon}}(u,v) \Big|\mathcal{A}^c(v|_r^-)\Big) \\
\le C_\gamma \cdot \mathbf{1}^{(2)}_{w,u}(\bar q,\epsilon)\cdot \left\{
\begin{array}{ll}
e^{r(\xi(\bar{q})+6\epsilon-(\widetilde{\xi}(\bar{q})+6\epsilon)(1-\gamma))}, & \text{ if } \gamma>1 \\
e^{r(\xi(\bar{q})+6\epsilon-\mathbf{1}_{\{p\ge n\}}\cdot(\xi(\bar{q})-6\epsilon)(1-\gamma))}, & \text{ if } \gamma<1
\end{array}
\right.
\end{multline*}
\begin{multline*}
\mathbf{1}^{(2)}_{w,u}(\bar q,\epsilon)\cdot \mathbb{E}\Big( \overline{\Delta \mathcal{K}^{S}_{\bar q,\delta+\epsilon}}(uu',vv') \Big|\mathcal{A}^c(v|_r^-)\Big) \\
\le C_\gamma \cdot \mathbf{1}^{(2)}_{w,u}(\bar q,\epsilon) \cdot\left\{
\begin{array}{ll}
e^{r(\widetilde{\xi}(\bar{q})+6\epsilon)(1+\gamma)-m(\xi(\bar{q})-6\epsilon)}, & \text{ if } \gamma>1 \\
e^{r(\xi(\bar{q})+6\epsilon)-m(\xi(\bar{q})-6\epsilon)(1-\gamma)}, & \text{ if } \gamma<1
\end{array}
\right.
\end{multline*}
\end{lemma}

To complete the proof, it remains to count the average number of pairs $(u,v)$ in $(\mathscr{A}^m)^2$ such that $\mathbf{1}_p(u,v)=1$ and $\mathbf{1}^{(2)}_{u,v}(\bar q,\epsilon)=1$. This is done in the next lemma, whose proof is given in Section~\ref{lemcounting}.

\begin{lemma}\label{counting}
For $m\ge 3r$, we have
$$
\mathbb{E}\Big(\sum_{u,v\in \mathscr{A}^m}  \mathbf{1}_p(u,v)\cdot \mathbf{1}^{(2)}_{u,v}(\bar q,\epsilon)\Big) \le 2b^{r-p+1}\cdot e^{(2m-r)(\gamma(\bar{q})+8\rho_K\epsilon)}.
$$
\end{lemma}

Now, by using Remark~\ref{J12} and the definition of $\gamma$, i.e. $\gamma=\gamma^{S}(\bar{q})-\delta-\epsilon$ for $S=G,R$, we have to deal with the following three cases (i), (ii), (iii):
\begin{equation*}
1-\gamma=\left\{
\begin{array}{llr}
\big(\gamma(\bar q)-\xi(\bar q)\big)/\widetilde{\xi}(\bar q)+\delta+\epsilon, & \text{if } \gamma>1, & \text{case (i)}\\
\delta+\epsilon, & \text{if } \gamma<1 \text{ and } K\subset J_1, & \text{case (ii)}\\
\big(\xi(\bar q)-\gamma(\bar q)\big)/\xi(\bar q) +\delta+\epsilon, & \text{if } \gamma<1 \text{ and } K\subset J_2 \text{ or } J_3, & \text{case (iii)}
\end{array}
\right.
\end{equation*}
Then, due to~\eqref{S},~\eqref{T},~\eqref{condiineq1} and~\eqref{condiineq2}, since $\mathbf{1}^{(2)}_{u,v}(\bar q,\epsilon)$, $Y_K(u)$ and $Y_K(v)$ (resp.  $Y_K(uu')$ and $Y_K(vv')$) are independent, taking the expectation of $Y_K(u)$ and $Y_K(v)$ (resp. $Y_K(uu')$ and $Y_K(vv')$), and using Lemmas~\ref{kernelcondi} and~\ref{counting}, for cases (i), (ii), (iii) we have ($C_K$ stands for $\mathbb{E}(Y_K)$, which is finite by Proposition~\ref{aumeas}(b)):
\begin{eqnarray*}
\mathbb{E}\left((\mathcal{I}^{S,\delta}_{n,\bar{q},\epsilon})_{p,m}\right) &\le& 2C_\gamma C_K^2 \cdot e^{-2m(\gamma(\bar{q})-4\rho_K\epsilon)}\cdot b^{r-p+1}\cdot e^{(2m-r)(\gamma(\bar{q})+8\rho_K\epsilon)}\\
&& \cdot \left\{
\begin{array}{ll}
e^{r(\xi(\bar{q})+6\epsilon-(\widetilde{\xi}(\bar{q})+6\epsilon)(\frac{\gamma(\bar q)-\xi(\bar q)}{\widetilde{\xi}(\bar q)}+\delta+\epsilon))}, & \text{(i)} \\
e^{r\mathbf{1}_{\{p< n\}}\cdot(\xi(\bar{q})-6\epsilon)(1-\gamma)}e^{r(\xi(\bar{q})+6\epsilon-(\xi(\bar{q})-6\epsilon)(\delta+\epsilon))}, & \text{(ii)}\\
e^{r\mathbf{1}_{\{p< n\}}\cdot(\xi(\bar{q})-6\epsilon)(1-\gamma)}e^{r(\xi(\bar{q})+6\epsilon-(\xi(\bar{q})-6\epsilon)(\frac{\xi(\bar q)-\gamma(\bar q)}{\xi(\bar q)} +\delta+\epsilon))}, & \text{(iii)}
\end{array}
\right.\\
&=& 2C_\gamma C_K^2\cdot b^{r-p+1} \cdot e^{\mathbf{1}_{\{p< n\}}[(\xi(\bar{q})-6\epsilon)(1-\gamma)\vee 0]\cdot r}\\
&& \cdot \left\{
\begin{array}{ll}
e^{-(\widetilde{\xi}(\bar q) \delta-[6-8\rho_K-\widetilde{\xi}(\bar q)-6(\frac{\gamma(\bar q)-\xi(\bar q)}{\widetilde{\xi}(\bar q)}+\delta+\epsilon)]\epsilon)\cdot r+24\rho_K\epsilon\cdot m}, & \text{(i)} \\
e^{-(\gamma(\bar q)-\xi(\bar q)+\xi(\bar q)\delta-[6-8\rho_K-\xi(\bar q)+6(\delta+\epsilon)]\epsilon)\cdot r+24\rho_K\epsilon\cdot m}, & \text{(ii)}\\
e^{-(\xi(\bar q)\delta-[6-8\rho_K-\xi(\bar q)+6(\frac{\xi(\bar q)-\gamma(\bar q)}{\xi(\bar q)} +\delta+\epsilon)]\epsilon)\cdot r+24\rho_K\epsilon\cdot m}, & \text{(iii)}
\end{array}
\right.\\
\mathbb{E}\left((\Delta \mathcal{I}^{S,\delta}_{n,\bar{q},\epsilon})_{p,m}\right) &\le& 2C_\gamma C_K^2 \cdot e^{-2(m+1)(\gamma(\bar{q})-4\rho_K\epsilon)}\cdot b^2\cdot b^{r-p+1}\cdot e^{(2m-r)(\gamma(\bar{q})+8\rho_K\epsilon)}\\
&& \cdot \left\{
\begin{array}{ll}
e^{r(\widetilde{\xi}(\bar{q})+6\epsilon)(1+\gamma)-m(\xi(\bar{q})-6\epsilon)}, & \text{(i)} \\
e^{r(\xi(\bar{q})+6\epsilon)-m(\xi(\bar{q})-6\epsilon)(1-\gamma)}, & \text{(ii) and (iii)}
\end{array}
\right.\\
&=& 2C_\gamma C_K^2 e^{-2(\gamma(\bar{q})-4\rho_K\epsilon)} b^2\cdot b^{r-p+1} \\
&&\cdot \left\{
\begin{array}{ll}
e^{[(\widetilde{\xi}(\bar{q})+6\epsilon)(1+\gamma)-\gamma(\bar{q})-8\rho_K\epsilon]\cdot r-(\xi(\bar{q})+18\epsilon)\cdot m}, & \text{(i)} \\
e^{[\xi(\bar{q})+6\epsilon-\gamma(\bar{q})-8\rho_K\epsilon]\cdot r-((\xi(\bar{q})-6\epsilon)(1-\gamma)+24\rho_K\epsilon)\cdot m}. & \text{(ii) and (iii)}
\end{array}
\right.
\end{eqnarray*}

Let $\eta_K=\inf_{q\in K} \xi(q)\wedge \widetilde{\xi}(q)$. Under our assumptions we have $\eta_K>0$. Let
\begin{equation*}
\begin{cases}
\kappa_1=\sup_{q\in K} \max\{ {\scriptstyle 6+8\rho_K+\widetilde{\xi}(q)+6\frac{|\gamma(q)-\xi(q)|}{\widetilde{\xi}(q)}+12, 6+8\rho_K+\xi(q)+6\frac{|\xi(q)-\gamma(q)|}{\xi(q)}+12, \xi(q)+1,24\rho_K}\},\\
\kappa_2=\sup_{q\in K} \max\{ {\scriptstyle 3(\widetilde{\xi}(q)+6)+\gamma(q)+8\rho_K, \xi(q)+\gamma(q)+6+8\rho_K}\},\\
\epsilon_*=\frac{\eta_K}{2\kappa_1+24(\rho_K\vee 1)}\wedge \epsilon_K,\ \eta_1=\frac{\eta_K}{2},\ \eta_2=\frac{\eta_K\delta}{2},\\
C=2C_\gamma C_K^2 b^2.
\end{cases}
\end{equation*}
Clearly those parameters are all positive and finite. Notice that
\[
\left\{
\begin{array}{ll}
1-\gamma\ge \delta, & \text{in case (ii) or (iii) };\\
\tau^*(\tau'(\bar q))/\tau'(\bar q)\ge 1 \text{ thus } \gamma(\bar q)-\xi(\bar q)\ge 0, & \text{in the case (ii) };\\
1+\gamma\le 3, \delta_K\vee \epsilon_K<1 & \text{in all cases}.
\end{array}
\right.
\]
Then, by construction, we get for any $\delta<\delta_K$ and $\epsilon<\epsilon_*$,
\begin{eqnarray*}
\mathbb{E}\left((\mathcal{I}^{S,\delta}_{n,\bar{q},\epsilon})_{p,m}\right) &\le& C\cdot b^{r -p+1} \cdot e^{\mathbf{1}_{\{p<n\}}\kappa_1\cdot r}\cdot e^{-\eta_1\delta \cdot r+\kappa_1\epsilon\cdot m} \\\text{and} \quad\mathbb{E}\left((\Delta \mathcal{I}^{S,\delta}_{n,\bar{q},\epsilon})_{p,m}\right) &\le& C\cdot b^{r -p+1} \cdot  e^{\kappa_2\cdot r-\eta_2\cdot m},
\end{eqnarray*}
which gives the conclusion. \end{proof}

\subsection{Proof of Lemma~\ref{kernelcondi}.}\label{proofpropker} $\ $\\

\noindent
\textbf{Step 1.} At first, we prove that the probability distribution of $Z_W=F_W(1)$ has  a bounded density function $f_W$, with $\|f_W\|_\infty=C_W<\infty$.

Let $\phi(t)=\mathbb{E}(e^{itZ_W})$ be the characteristic function of $Z_W$. Since we have $Z_W=\sum_{j=0}^{b-1} W_j \cdot Z_W(j)$, where $\{W_j\}_j$ and $\{Z_W(j)\}_j$ are independent, and the $Z_W(j)$ are independent copies of $Z_W$, we have
$$
\phi(t)  = \mathbb{E}\left(\mathbb{E}\left(e^{it\cdot \sum_{j=0}^{b-1}W_j\cdot Z_W(j)}\ \Big|\ \sigma(Z_W(j), 0\le j\le b-1)\right)\right) = \mathbb{E}\left(\prod\nolimits_{j=0}^{b-1}\phi(W_jt)\right).
$$
Since $\mathbb{E}(|Z_W|)<\infty$, simultaneously we also get
$$
\phi'(t)= \mathbb{E}(iZ_W\cdot e^{itZ_W})=\mathbb{E}\left(\sum\nolimits_{k=0}^{b-1}i W_k\phi'(W_kt)\prod\nolimits_{j\neq k}\phi(W_jt) \right).
$$
Notice that $|\phi(t)|=|\phi(-t)|$ and $|\phi'(t)|=|\phi'(-t)|$, so we have $|\phi(|t|)|=|\phi(t)|=|\phi(-t)|$ and $|\phi'(|t|)|=|\phi'(t)|=|\phi'(-t)|$, then
\begin{eqnarray}
|\phi(t)|&\leq&\mathbb{E}\left(\prod\nolimits_{j=0}^{b-1}|\phi(W_j t)|\right)=\mathbb{E}\left(\prod\nolimits_{j=0}^{b-1}|\phi(|W_j|t)|\right);\label{eqn:1}\\
|\phi'(t)|&\leq&\mathbb{E}\left(\sum\nolimits_{k=0}^{b-1} |W_k||\phi'(|W_k|t)|\prod\nolimits_{j\neq k}|\phi(|W_j|t)| \right);\label{eqn:2}
\end{eqnarray}

\noindent
Define $l=\limsup_{t\rightarrow\infty}|\phi(t)|$. Since $|\phi(t)|\leq 1$, we have $l\leq 1$. From Fatou's lemma and the fact that $\mathbb{P}(\forall\ j, W_j \neq 0)=1$, we have
$$
l \leq \limsup_{t\rightarrow\infty}\mathbb{E}\left(\prod\nolimits_{j=0}^{b-1}|\phi(|W_j|t)|\right)\leq \mathbb{E}\left(\limsup_{t\rightarrow\infty}\prod\nolimits_{j=0}^{b-1}|\phi(|W_j|t)|\right) = l^b.
$$
This implies that $l=0$ or $1$. Since we are in the non-conservative case, $Z_W$ is not almost surely a constant. Consequently, we can use the same approach as in  the proof of Lemma 3.1 in~\cite{Liu} (which deals with the case $W\ge 0$), and using the fact that $\mathbb{E}(\max_{0\leq j\leq b-1} |W_j|^p)\le \mathbb{E}(\sum_{i=0}^{b-1} |W_i|^p)<1$ for some $p> 1$, we obtain $l=0$.

Define $N=\min_{0\leq i\leq b-1,W_i\ne 0} |W_i|$. Due to assumption (A3) there exists a $q>1$ such that $\mathbb{E}(N^{-q})\leq\mathbb{E}(\sum_{i=0}^{b-1}\mathbf{1}_{\{W_i\ne 0\}}|W_i|^{-q})<\infty$. Then by using~\eqref{eqn:1},~\eqref{eqn:2}, the same arguments as in the proofs of Theorem 2.1 and 2.2 in~\cite{Liu} we can get $|\phi(t)|=O(t^{-q})$ and $|\phi'(t)|=O(t^{-(q+1)})$ when $t\rightarrow\infty$. Now as a consequence (Lemma 3 in~\cite{Athreya}) we have that $Z_W$ has a density function, which is bounded by $\int_\mathbb{R} |\phi(t)|\mathrm{d}t<\infty$. 

\medskip

\noindent
\textbf{Step 2.} Recall that $\gamma=\gamma^{S}(\bar{q})-\delta-\epsilon$ and $\mathcal{K}^{S}_{\bar q,\delta+\epsilon}(u,v)=\mathcal{K}_\gamma(u,v)$, as well as
\begin{equation*}
\left\{
\begin{array}{l}
\overline{\mathcal{K}^{S}_{\bar q,\delta+\epsilon}}(u,v)=\mathcal{K}_\gamma(u,v)\cdot \mathbf{1}^{(1)}_{u,v}(\bar q,\epsilon):=\overline{\mathcal{K}_\gamma}(u,v); \\
\ \\
\overline{\Delta \mathcal{K}^{S}_{\bar q,\delta+\epsilon}}(uu',vv')=\big|\mathcal{K}_\gamma(uu',vv')-\mathcal{K}_\gamma(u,v)\big|\cdot \mathbf{1}^{(1)}_{u,v}(\bar q,\epsilon):=\overline{\Delta \mathcal{K}_\gamma}(uu',vv').
\end{array}
\right.
\end{equation*}

Let us prove the desired estimates, i.e., there exists a constant $C_\gamma>0$ such that
\begin{multline*}
\mathbf{1}^{(2)}_{w,u}(\bar q,\epsilon)\cdot \mathbb{E}\Big(\overline{\mathcal{K}_\gamma}(u,v)\Big| \mathcal{A}^c(v|_r^-)\Big) \\
\le C_\gamma\cdot \mathbf{1}^{(2)}_{w,u}(\bar q,\epsilon) \cdot \left\{
\begin{array}{ll}
e^{r(\xi(\bar{q})+6\epsilon-(\widetilde{\xi}(\bar{q})+6\epsilon)(1-\gamma))}, & \text{ if } \gamma>1; \\
e^{r(\xi(\bar{q})+6\epsilon-\mathbf{1}_{\{p\ge n\}}\cdot(\xi(\bar{q})-6\epsilon)(1-\gamma))}, & \text{ if } \gamma<1;
\end{array}
\right.
\end{multline*}
\begin{multline*}
\mathbf{1}^{(2)}_{w,u}(\bar q,\epsilon)\cdot \mathbb{E}\Big(\overline{\Delta \mathcal{K}_\gamma}(uu',vv')\Big| \mathcal{A}^c(v|_r^-) \Big) \\
\le C_\gamma\cdot \mathbf{1}^{(2)}_{w,u}(\bar q,\epsilon)\cdot \left\{
\begin{array}{ll}
e^{r(\widetilde{\xi}(\bar{q})+6\epsilon)(1+\gamma)-m(\xi(\bar{q})-6\epsilon)}, & \text{ if } \gamma>1; \\
e^{r(\xi(\bar{q})+6\epsilon)-m(\xi(\bar{q})-6\epsilon)(1-\gamma)}, & \text{ if } \gamma<1.
\end{array}
\right.
\end{multline*}

The $\sigma$-algebra $\mathcal{A}^c(v|_r^-)$ being defined as in \textbf{Claim 2}, we simplify the notations of the following quantities, which are measurable with respect to $\mathcal{A}^c(v|_r^-)$, hence constant given $\mathcal{A}^c(v|_r^-)$: 
\begin{equation*}
\begin{cases}
A=W_{v|_r^-};\\
B=F_W(v)-F_W(v|_r)+F_W(v|_r^-)-F_W(u);\\
C_1=\left\{\begin{array}{ll}1, & \text{ if } p<n;\\ 2e^{-p(\xi(\bar{q})-6\epsilon)}, & \text{ if } p\ge n, \end{array}\right.;\\
C_2=e^{-(n\vee p)(\widetilde{\xi}(\bar{q})+6\epsilon)};\\
D_1=F_W(vv')-F_W(uu')-(F_W(v)-F_W(u));\\
D_2=F_L(vv')-F_L(uu')-(F_L(v)-F_L(u)).
\end{cases}
\end{equation*}

Let $f_W$ stand for the bounded density  of $Z_W(v_r^-)$ obtained in \textbf{Step 1}. Let
$$g_\gamma(s,t)=\Big(|F_W(v)-F_W(u)+s|^2+|F_L(v)-F_L(u)+t|^2\Big)^{-\gamma/2},\ s,t\in\mathbb{R}.$$

Define
\begin{eqnarray*}
\zeta_1(\gamma) &=&  \int_{\mathbb{R}}\frac{f_W(x)}{(|Ax+B|^2+C_2^2)^{\gamma/2}}\mathrm{d}x;\\
\zeta_2(\gamma) &=& \int_{|t|\le |D_2|} \Big|\frac{\partial }{\partial t}g_\gamma(0,t)\Big| \mathrm{d} t + \int_{|s|\le |D_1|} \sup_{|t|\le |D_2|}\Big|\frac{\partial }{\partial s}g_\gamma(s,t)\Big|\mathrm{d}s;\\
\zeta_3(\gamma) &=& \int_{|Ax+B| \le C_1}\frac{f_W(x)}{|Ax+B|^{\gamma}}  \mathrm{d}x;\\
\zeta_4(\gamma) &=& \int_{\mathbb{R}}\Big|\frac{1}{|Ax+B+D_1|^{\gamma}}-\frac{1}{|Ax+B|^{\gamma}}\Big| f_W(x) \mathrm{d}x.
\end{eqnarray*}

From~\eqref{chi1} and \textbf{Claim 2}, we have
\begin{equation*}
\begin{cases}
\mathbf{1}^{(1)}_{u,v}(\bar q,\epsilon)\cdot \Big|F_W(v)-F_W(u)\Big| \wedge 1 \le C_1;\\
\mathbf{1}^{(1)}_{u,v}(\bar q,\epsilon)\cdot  \Big(F_L(v)-F_L(u)\Big)\ge \mathbf{1}^{(1)}_{u,v}(\bar q,\epsilon)\cdot C_2.
\end{cases}
\end{equation*}
This implies
\begin{eqnarray}
\mathbf{1}^{(2)}_{u,v}(\bar q,\epsilon)\cdot \mathbb{E}\Big(\overline{\mathcal{K}_\gamma}(u,v)\Big| \mathcal{A}^c(v|_r^-)\Big) &\le& \mathbf{1}^{(2)}_{u,v}(\bar q,\epsilon)\cdot \left\{
\begin{array}{ll}
1+\zeta_1(\gamma), & \text{ if } \gamma>1; \\
1+\zeta_3(\gamma), & \text{ if } \gamma<1;
\end{array}
\right.\label{con1}\\
\mathbf{1}^{(2)}_{u,v}(\bar q,\epsilon)\cdot \overline{\Delta \mathcal{K}_\gamma}(u,v) &\le & \mathbf{1}^{(2)}_{u,v}(\bar q,\epsilon) \cdot \zeta_2(\gamma), \text{ if } \gamma>1; \label{con2.1}\\
\mathbf{1}^{(2)}_{u,v}(\bar q,\epsilon)\cdot \mathbb{E}\Big(\overline{\Delta \mathcal{K}_\gamma}(u,v)\Big| \mathcal{A}^c(v|_r^-) \Big) &\le& \mathbf{1}^{(2)}_{u,v}(\bar q,\epsilon)\cdot \zeta_4(\gamma), \text{ if } \gamma<1,
\label{con2.2}
\end{eqnarray}
where we have used the inequality $|x\vee 1-y\vee 1|\le |x-y|$ holds for any $x,y\ge 0$.

\medskip

Now, we have the following inequalities:
\medskip

\begin{itemize}
\item[(I)] By using the change of variable $y=\frac{Ax+B}{C_2}$ we get
$$ \mathbf{1}^{(2)}_{w,u}(\bar q,\epsilon)\cdot \zeta_1(\gamma) \le C_W|A|^{-1}C_2^{1-\gamma}\cdot \int_{\mathbb{R}}\frac{\mathrm{d}y}{(y^2+1)^{\frac{\gamma}{2}}};$$
\item[(II)] It is not difficult to check that when $F_L(v)-F_L(u) \ge C_2$ and $|t|\le |D_2|$ we always have
\[
\Big|\frac{\partial }{\partial t}g_\gamma(0,t)\Big|\vee \Big|\frac{\partial }{\partial s}g_\gamma(s,t)\Big| \le \gamma\cdot |F_L(v)-F_L(u)+t|^{-\gamma-1}\le \gamma\cdot((C_2-|D_2|)\vee 0)^{-\gamma-1}.
\]
In fact, we have
\begin{eqnarray*}
\Big|\frac{\partial }{\partial t}g_\gamma(0,t)\Big| &\le &\frac{\gamma\cdot |F_L(v)-F_L(u)+t|}{\Big(|F_W(v)-F_W(u)|^2+|F_L(v)-F_L(u)+t|^2\Big)^{\frac{\gamma}{2}+1}}\\
&\le &\frac{\gamma\cdot |F_L(v)-F_L(u)+t|}{|F_L(v)-F_L(u)+t|^{\gamma+2}}=\gamma\cdot |F_L(v)-F_L(u)+t|^{-\gamma-1}\\
\Big|\frac{\partial }{\partial s}g_\gamma(s,t)\Big| &\le &\frac{\gamma\cdot |F_W(v)-F_W(u)+s|}{\Big(|F_W(v)-F_W(u)+s|^2+|F_L(v)-F_L(u)+t|^2\Big)^{1+\frac{\gamma}{2}}}\\
&\le &\frac{\gamma\cdot |F_W(v)-F_W(u)+s| \cdot |F_L(v)-F_L(u)+t|^{-\gamma}}{|F_W(v)-F_W(u)+s|^2+|F_L(v)-F_L(u)+t|^2}\\
&\le & \frac{\gamma\cdot |F_L(v)-F_L(u)+t|^{-\gamma} }{2|F_L(v)-F_L(u)+t|} \le \gamma\cdot |F_L(v)-F_L(u)+t|^{-\gamma-1}
\end{eqnarray*}
(where we have used that for any $a,b>0$, $\frac{a}{a^2+b^2} \le \frac{1}{2b}$). This together with the definition of $\zeta_2(\gamma)$ yields 
\begin{equation}
\mathbf{1}^{(2)}_{u,v}(\bar q,\epsilon)\cdot \zeta_2(\gamma)\le 2((C_2-|D_2|)\vee 0)^{-\gamma-1}(|D_1|+|D_2|);\label{zeta2}
\end{equation}
\item[(III)] By using the change of variable $y=Ax+B$ when $\gamma< 1$ we get
$$
 \mathbf{1}^{(2)}_{u,v}(\bar q,\epsilon) \cdot \zeta_3(\gamma) \le C_W|A|^{-1}\int_{|u|\le C_1}\frac{\mathrm{d}y}{|y|^\gamma}=2C_W|A|^{-1}C_1^{1-\gamma};
$$
\item[(IV)] By using the change of variable $y=\frac{Ax+B}{D_1}$ we get
$$\mathbf{1}^{(2)}_{u,v}(\bar q,\epsilon) \cdot \zeta_4(\gamma) \le C_W \frac{|D_1|^{1-\gamma}}{|A|}\int_{\mathbb{R}} \Big|\frac{1}{|y+1|^{\gamma}}-\frac{1}{|y|^{\gamma}}\Big| \mathrm{d}y.$$
\end{itemize}

Now we notice that
$$
\int_{\mathbb{R}}(y^2+1)^{-\gamma/2}\mathrm{d}y \ (\gamma>1) \text{ and } \int_{\mathbb{R}} \Big|\frac{1}{|y+1|^{\gamma}}-\frac{1}{|y|^{\gamma}}\Big| \mathrm{d}y \ (\gamma<1)
$$
are both finite and when $\mathbf{1}^{(2)}_{u,v}(\bar q,\epsilon)=1$, we have
\begin{equation*}
\begin{cases}
|A|^{-1}=|W_{v|_r^-}|^{-1}\le b^{r(\xi(\bar{q})+2\epsilon)},\\
D_2\le \mathrm{Osc}_{F_L}(I_v)+\mathrm{Osc}_{F_L}(I_u)\le 2b^{-m(\widetilde{\xi}(\bar{q})-6\epsilon)},\\
D_1\le \mathrm{Osc}_{F_W}(I_v)+\mathrm{Osc}_{F_W}(I_u) \le 2b^{-m(\xi(\bar{q})-6\epsilon)}.
\end{cases}
\end{equation*}
Moreover, when $\gamma>1$ we have $\xi(\bar{q})<\widetilde{\xi}(\bar{q})$, so if $m\ge 3r$ and $\widetilde{\xi}(\bar{q})-12\epsilon>0$ then
\begin{eqnarray*}
\zeta_2(\gamma)&\le& [b^{-r(\widetilde{\xi}(\bar{q})+6\epsilon)}(1-b^{-(m-r)(\widetilde{\xi}(\bar{q})-6\frac{m+r}{m-r}\epsilon)})]^{-\gamma-1}2(b^{-m(\xi(\bar{q})-6\epsilon)}+b^{-m(\widetilde{\xi}(\bar{q})-6\epsilon)})\\
&\le& 4\cdot b^{r(\widetilde{\xi}(\bar{q})+6\epsilon)(1+\gamma)}\cdot b^{-m(\xi(\bar{q})-6\epsilon)}.
\end{eqnarray*}
Then by applying these inequalities to~\eqref{con1},~\eqref{con2.1},~\eqref{con2.2} we get the conclusion. 

\begin{remark}\label{diffkernel2}
At the end of the proof of Theorem~\ref{levelsets}, we define two quantities $
\Delta \mathcal{K}_{\widetilde{\gamma}}(uu',vv')$ and $\widetilde{(\Delta \mathcal{I}^{S,\delta}_{n,\bar{q},\epsilon})}_{p,m}$, and we claim~\eqref{boundenergy}. We justify this claim here. In fact, $(\Delta \mathcal{I}^{S,\delta}_{n,\bar{q},\epsilon})_{p,m}$ and $\widetilde{(\Delta \mathcal{I}^{S,\delta}_{n,\bar{q},\epsilon})}_{p,m}$ can be estimated from above similarly: we have (the first inequality is similar to~\eqref{con2.1})
\begin{eqnarray*}
\mathbf{1}^{(1)}_{u,v}(\bar q,\epsilon) \cdot\mathbf{1}^{(2)}_{u,v}(\bar q,\epsilon)\cdot \Delta \mathcal{K}_{\widetilde{\gamma}}(uu',vv') &\le& C_2^{-\widetilde{\gamma}}\cdot \zeta_2(1)+C_2^{-1}\zeta_2(\widetilde{\gamma}) \\
&\le& 2(C_2-|D_2|)^{-2-\widetilde{\gamma}}(|D_1|+|D_2|),
\end{eqnarray*}
which is exactly the same bound as in (II) for $\mathbf{1}^{(2)}_{w,u}(\bar q,\epsilon)\cdot \zeta_2(\gamma)$ (we have used \eqref{zeta2} with $\gamma=1$ and $\gamma=\widetilde{\gamma}$). Then, the upper bound estimation of $\widetilde{(\Delta \mathcal{I}^{S,\delta}_{n,\bar{q},\epsilon})}_{p,m}$ can be treated like that of  $(\Delta \mathcal{I}^{S,\delta}_{n,\bar{q},\epsilon})_{p,m}$, and one obtains the same estimate as for $\mathbb{E}\left((\Delta \mathcal{I}^{S,\delta}_{n,\bar{q},\epsilon})_{p,m}\right)$:
$$
\mathbb{E}\left(\widetilde{(\Delta \mathcal{I}^{S,\delta}_{n,\bar{q},\epsilon})}_{p,m}\right) \le C\cdot b^{r -p+1} \cdot  e^{\kappa_2\cdot r-\eta_2\cdot m},
$$
which is enough to get~\eqref{boundenergy}.
\end{remark}

\subsection{Proof of Lemma~\ref{counting}.}\label{lemcounting}

\begin{proof}
Let
\begin{multline*}
S_{p,m}=\sum_{u,v\in \mathscr{A}^m} \mathbf{1}_p(u,v)\cdot  \mathbf{1}_{\mathscr{W}_{u|_r}(\bar{q},2\epsilon)}\mathbf{1}_{\mathscr{W}_{u}(\bar{q},2\epsilon)} \cdot \mathbf{1}_{\mathscr{W}_{v|_r}(\bar{q},2\epsilon)}\mathbf{1}_{\mathscr{W}_{v}(\bar{q},2\epsilon)} \cdot \\
  \mathbf{1}_{\mathscr{L}_{u|_r}(\bar{q},2\epsilon)} \mathbf{1}_{\mathscr{L}_{u}(\bar{q},2\epsilon)} \cdot \mathbf{1}_{\mathscr{L}_{v|_r}(\bar{q},2\epsilon)} \mathbf{1}_{\mathscr{L}_{v}(\bar{q},2\epsilon)}
\end{multline*}
Then by~\eqref{chi2} we have $\sum_{u,v\in \mathscr{A}^m} \mathbf{1}_p(u,v)\cdot\mathbf{1}^{(2)}_{u,v}(\bar q,\epsilon)\le S_{p,m}$.

Recall that $r=p\vee n$ and $m\ge 3r$. For any $u\in\mathscr{A}^m$ we write $u=u|_r\cdot u'$ with $u'\in\mathscr{A}^{m-r}$. From~(\ref{Wij}) we have $W_{u}=W_{u|_r} \cdot W_{u'}(u|_r)$, so $\mathbf{1}_{\mathscr{W}_{u|_r}{(\bar{q},2\epsilon)}}\cdot\mathbf{1}_{ \mathscr{W}_{u}{(\bar{q},2\epsilon)}}=1$ implies that
$$|W_{u'}(u|_r)| \in [e^{-(m-r)(\xi(\bar{q})+2\frac{m+r}{m-r}\epsilon)},e^{-(m-r)(\xi(\bar q)-2\frac{m+r}{m-r}\epsilon)}].$$
Thus, when $m\ge 3r$, we have
$$\mathbf{1}_{\mathscr{W}_{u|_r}{(\bar{q},2\epsilon)}}\cdot\mathbf{1}_{ \mathscr{W}_{u}{(\bar{q},2\epsilon)}} \le \mathbf{1}_{\mathscr{W}_{u|_r}{(\bar{q},2\epsilon)}}\cdot\mathbf{1}_{ \mathscr{W}_{u'}^{[u|_r]}(\bar{q},4\epsilon)},$$
and, moreover, $\mathbf{1}_{\mathscr{W}_{u|_r}{(\bar{q},2\epsilon)}}$ and $\mathbf{1}_{ \mathscr{W}_{u'}^{[u|_r]}(\bar{q},4\epsilon)}$ are independent. Simultaneously we also have
$$\mathbf{1}_{\mathscr{W}_{v|r}{(\bar{q},2\epsilon)}}\cdot\mathbf{1}_{ \mathscr{W}_{v}{(\bar{q},2\epsilon)}} \le \mathbf{1}_{\mathscr{W}_{v|_r}{(\bar{q},2\epsilon)}}\cdot\mathbf{1}_{ \mathscr{W}_{v'}^{[v|_r]}(\bar{q},4\epsilon)},$$
$$\mathbf{1}_{\mathscr{L}_{u|r}{(\bar{q},2\epsilon)}}\cdot\mathbf{1}_{ \mathscr{L}_{u}{(\bar{q},2\epsilon)}} \le \mathbf{1}_{\mathscr{L}_{u|_r}{(\bar{q},2\epsilon)}}\cdot\mathbf{1}_{ \mathscr{L}_{u'}^{[u|_r]}(\bar{q},4\epsilon)},$$
$$\mathbf{1}_{\mathscr{L}_{v|r}{(\bar{q},2\epsilon)}}\cdot\mathbf{1}_{ \mathscr{L}_{v}{(\bar{q},2\epsilon)}} \le \mathbf{1}_{\mathscr{L}_{v|_r}{(\bar{q},2\epsilon)}}\cdot\mathbf{1}_{ \mathscr{L}_{v'}^{[v|_r]}(\bar{q},4\epsilon)}.$$

We can drop the terms $\mathbf{1}_{\mathscr{W}_{v|r}{(\bar{q},2\epsilon)}}$ and $\mathbf{1}_{\mathscr{L}_{v|r}{(\bar{q},2\epsilon)}}$ so that the remaining indicator functions on the right hand side of the above inequalities are independent. Since for each $u\in\mathscr{A}^m$, there are at most $2b^{r-p+1}$ many $v|_r$ such that $\mathbf{1}_p(u,v)=1$, we get
\begin{eqnarray*}
\mathbb{E}(N_{p,m}) &\le& 2b^{r-p+1}\mathbb{E}\big(\sum_{u\in \mathscr{A}^{r}} \mathbf{1}_{\mathscr{W}_{u}(\bar{q},2\epsilon)}\mathbf{1}_{\mathscr{L}_{u}(\bar{q},2\epsilon)}\big) \cdot \mathbb{E}\big(\sum_{u\in \mathscr{A}^{m-r}}\mathbf{1}_{\mathscr{W}_{u}(\bar{q},4\epsilon)}\mathbf{1}_{\mathscr{L}_{u}(\bar{q},4\epsilon)}\big)^2.
\end{eqnarray*}
For $\bar{q}$ with $|\bar q|+\xi(\bar{q})+\widetilde{\xi}(\bar{q})+|\tau(\bar q)|\le \rho_K$ we always have
$$
\mathbf{1}_{\mathscr{W}_{u}(\bar{q},2\epsilon)} \le |W_{u}|^{\bar q}\cdot e^{|u|(\bar{q}\xi(\bar{q})+2\rho_K\epsilon)} \text{ and } \mathbf{1}_{\mathscr{L}_{u}(\bar{q},2\epsilon)} \le L_{u}^{\tau(\bar q)}\cdot e^{|u|(-\tau(\bar q)\widetilde{\xi}(\bar{q})+2\rho_K\epsilon)}.
$$
For $k\in \{r,m-r\}$ this yields
$$
\mathbb{E}\left(\sum_{u\in \mathscr{A}^k} \mathbf{1}_{\mathscr{W}_{u}(\bar{q},4\epsilon)}\mathbf{1}_{\mathscr{L}_{u}(\bar{q},4\epsilon)} \right) \le e^{k(\gamma(\bar{q})+8\rho_K\epsilon)},
$$
where recall that $\gamma(\bar{q})=\bar{q}\xi(\bar{q})-\tau(\bar q)\widetilde{\xi}(\bar{q})$. This gives us the conclusion. 
\end{proof}

\medskip

\noindent
{\bf Acknowledgement.}

\medskip

\noindent
The author would like to thank gratefully his supervisor Julien Barral for having suggested him to study  the new types of singularity spectra considered in this paper for $b$-adic independent cascade function, as well as for his help in achieving this work.

\end{document}